\documentclass[11pt,twoside]{article}

\usepackage[utf8]{inputenc}
\usepackage[colorlinks=true
]{hyperref}
\hypersetup{urlcolor=blue, citecolor=red, linkcolor=blue}
\usepackage{mleftright}
\usepackage{stmaryrd}
\SetSymbolFont{stmry}{bold}{U}{stmry}{m}{n}
\usepackage{microtype}

\usepackage{bbm}
\usepackage{colortbl}

\usepackage{graphicx}
\usepackage[active]{srcltx}

\usepackage{amsmath,amssymb}

\usepackage{extarrows}

\numberwithin{equation}{section}

\usepackage{physics}
\usepackage{bm}
\usepackage{enumerate}
\usepackage{amsthm}
\usepackage[all,2cell]{xy}
\usepackage{mathrsfs}

\usepackage{mathtools}
\mathtoolsset{showonlyrefs}

\usepackage{tikz-cd}
\usetikzlibrary{cd}

\usepackage{xcolor}

\hypersetup{urlcolor=blue, citecolor=red, linkcolor=blue}

\usepackage[a4paper, left=25mm, right=25mm, top=30mm, bottom=25mm]{geometry}

\usepackage{marginnote}

\usepackage{imakeidx}
\makeindex[intoc]

\usepackage[nottoc]{tocbibind}
\usepackage[colorinlistoftodos]{todonotes} 

\theoremstyle{plain}
\newtheorem{theorem}{Theorem}[section]
\newtheorem{proposition}[theorem]{Proposition}
\newtheorem{corollary}[theorem]{Corollary}
\newtheorem{lemma}[theorem]{Lemma}

\theoremstyle{definition}
\newtheorem{definition}[theorem]{Definition}
\newtheorem{remark}[theorem]{Remark}
\newtheorem{example}[theorem]{Example}

\AtBeginEnvironment{remark}{%
  \pushQED{\qed}%
}
\AtEndEnvironment{remark}{\popQED\endexample}

\AtBeginEnvironment{example}{%
  \pushQED{\qed}%
}
\AtEndEnvironment{example}{\popQED\endexample}

\newcommand\restr[2]{{
  \left.\kern-\nulldelimiterspace #1 \right|_{#2} 
}}

\newcommand*{\transp}[2][-3mu]{\ensuremath{\mskip1mu\prescript{\smash{\mathsf t\mkern#1}}{}{\mathstrut#2}}}%

\newcommand{\R}{\mathbb{R}}

\renewcommand{\d}{\mathrm{d}}

\newcommand{\Cinfty}{\mathscr{C}^\infty}
\newcommand{\T}{\mathrm{T}}

\newcommand{\Id}{\mathrm{Id}}
\newcommand{\leftperp}[1]{\,^{\perp} #1}

\def\K{\mathcal{K}}
\newcommand*{\inn}[1]{\iota_{#1}}

\newcommand{\sksym}[1]{\mathsf{\Lambda}^{#1}}
\newcommand{\sym}[1]{\mathsf{S}^{#1}}

\DeclareMathOperator{\codim}{codim}

\newcommand{\Lie}{\mathscr{L}}

\newcommand{\X}{\mathfrak{X}}

\renewcommand{\H}{\mathcal{H}}

\newcommand{\D}{\mathcal{D}}
\newcommand{\V}{\mathcal{V}}

\DeclareMathOperator{\cl}{cl}

\DeclareMathOperator{\Ima}{Im}
\def\longto{\longrightarrow}

\DeclareMathAlphabet{\mathpzc}{OT1}{pzc}{m}{it}
\def\d{\mathrm{d}}

\DeclareMathOperator{\curv}{curv}

\DeclareMathOperator{\pr}{pr}
\DeclareMathOperator{\Ker}{Ker}

\newcommand{\dfn}[1]{\textit{\textbf{{#1}}}}

\usepackage{xcolor}
\usepackage{fancyhdr}

\begin{document}





{\Huge\sffamily\raggedright 
From 2-covariant tensor fields\\[3mm]
to Poisson and Jacobi structures
}
\vspace{1em}

{\large\sffamily\raggedright
    September 16, 2026
}

\vspace{1em}

{\large\raggedright\sffamily
    Manuel de León
}\vspace{1mm}\newline
{\raggedright
    Instituto de Ciencias Matemáticas (ICMAT--CSIC).\\
    C. Nicolás Cabrera, 13-15, Fuencarral, 28049 Madrid, Spain.
    \\[1.5pt]    
    Real Academia de Ciencias Exactas, Físicas y Naturales de España.\\
    C. Valverde, 22, 28004 Madrid, Spain.
    \\[1.5pt]
    Department of Mathematics, School of Sciences, Great Bay University, Dongguan, Guangdong, 523000 China.\\    
    e-mail: \href{mailto:mdeleon@icmat.es}{mdeleon@icmat.es} --- orcid: \href{https://orcid.org/0000-0002-8028-2348}{0000-0002-8028-2348}
}

\medskip

{\large\raggedright\sffamily
    Xavier Gràcia
}\vspace{1mm}\newline
{\raggedright
    Dept.\ of Mathematics, Universitat Politècnica de Catalunya, Barcelona\\
    e-mail: 
    \href{mailto:xavier.gracia@upc.edu}{xavier.gracia@upc.edu} --- orcid: 
    \href{https://orcid.org/0000-0003-1006-4086}{0000-0003-1006-4086}
}

\medskip

{\large\raggedright\sffamily
    Rubén Izquierdo-López
}\vspace{1mm}\newline
{\raggedright
    Instituto de Ciencias Matemáticas (ICMAT--CSIC).\\
    C. Nicolás Cabrera, 13-15, Fuencarral, 28049 Madrid, Spain.
    
    e-mail: \href{mailto:ruben.izquierdo@icmat.es}{ruben.izquierdo@icmat.es} --- orcid: \href{https://orcid.org/0009-0007-8747-344X}{0009-0007-8747-344X}
}

\medskip

{\large\raggedright\sffamily
    Ángel Martínez-Muñoz
}\vspace{1mm}\newline
{\raggedright
    Department of Computer Engineering and Mathematics, Universitat Rovira i Virgili\\
    Avinguda Països Catalans 26, 43007 Tarragona, Spain\\
    e-mail: \href{mailto:angel.martinezm@urv.cat}{angel.martinezm@urv.cat} --- orcid: \href{https://orcid.org/0009-0002-8944-5403}{0009-0002-8944-5403}
}

\medskip

{\large\raggedright\sffamily
    Xavier Rivas
}\vspace{1mm}\newline
{\raggedright
    Department of Computer Engineering and Mathematics, Universitat Rovira i Virgili\\
    Avinguda Països Catalans 26, 43007 Tarragona, Spain\\
    e-mail: \href{mailto:xavier.rivas@urv.cat}{xavier.rivas@urv.cat} --- orcid: \href{https://orcid.org/0000-0002-4175-5157}{0000-0002-4175-5157}
}

\vspace{1em}

{\large\bf\raggedright
    Abstract
}\vspace{1mm}\newline
{\raggedright
Poisson and Jacobi structures play a fundamental role in the geometric description of many systems arising in classical mechanics. In most cases, the corresponding bivector field is induced by a non-degenerate 2-covariant tensor field. In this paper, we present a unified framework for constructing the associated brackets by studying the Poisson and Jacobi structures induced by these tensors. More specifically, under suitable assumptions on the tensor, we derive a formula for computing the Schouten--Nijenhuis bracket of the associated bivector field in terms of the curvature of a certain distribution and the exterior derivative of a differential form. This formula provides the obstruction to the existence of a Poisson or Jacobi structure. To illustrate the theory, we recover the classical brackets associated with symplectic, locally conformally symplectic, contact, cosymplectic, and cocontact geometries. Finally, we characterise the conditions under which fat bundles and almost cosymplectic structures of order $p$ determine a Jacobi bracket.
}

\vspace{1em}

{\large\bf\raggedright
    Keywords:} Poisson structure, Jacobi structure, bracket, bilinear form, $2$-covariant tensor field, bivector field, Schouten--Nijenhuis bracket, curvature, cocontact structure, cosymplectic structure, fat bundle
\medskip

{\large\bf\raggedright
MSC2020 codes}:
53D17,
15A69,  
53D10,
53D05,
70G45
	

\medskip




%


{\setcounter{tocdepth}{3}
\def\baselinestretch{1}
\small
\def\addvspace#1{\vskip 1pt}
\parskip 0pt plus 0.1mm
\tableofcontents
}

\pagestyle{fancy}

\fancyhead[C]{}                  
\fancyhead[LE]{Poisson and Jacobi structures from 2-covariant tensors}    
\fancyhead[LO]{\thepage}                  
\fancyhead[RE]{\thepage}                  
\fancyhead[RO]{M. de León, X. Gràcia, R. Izquierdo-López, Á. Martínez-Muñoz and X. Rivas}       

\fancyfoot[L]{}     
\fancyfoot[C]{}                  
\fancyfoot[R]{}            

\renewcommand{\headrulewidth}{0.1pt}  
\renewcommand{\footrulewidth}{0pt}    

\renewcommand{\headrule}{%
    \vspace{3pt}                
    \hrule width\headwidth height 0.4pt 
    \vspace{0pt}                
}

\setlength{\headsep}{30pt}  

\section{Introduction}


The study of classical mechanics was revolutionized by the introduction of symplectic geometry (see, for instance, \cite{AM_78, Arn_89, LR_89, LM_87}), providing powerful differential geometric tools for the analysis of Hamiltonian systems. 
This framework has yielded profound results with vast applications, including Darboux's theorem, the Liouville--Arnold theorem for integrable systems, and KAM theory.

The success of symplectic geometry has naturally inspired the study and development of other geometric structures capable of modelling more complex dynamical systems. This effort has led to the formalisation (see, for example,~\cite{Alb_89,Bra_18,CG_2025,LL_19,LR_89,GGMRR_20a,MR_2025}) of contact geometry for describing time-dependent, dissipative and thermodynamic systems, cosymplectic geometry for non-autonomous mechanics, and locally conformally symplectic manifolds for non-conservative systems, among others. 
Beyond mathematical physics, these novel geometric structures have also found applications in modern topology (see~\cite{Gei_08,McDuff_2017}), serving as useful tools in areas such as knot theory and low-dimensional topology.

For these reasons, some of the present authors recently introduced in~\cite{GMR_2026} a general framework for studying geometric structures defined by a pair $(\tau, \omega)$ of a 1-form and a 2-form, under suitable regularity conditions (the so-called \textit{doublets}, see Section~\ref{sec:doublets}). As it turns out, most of the usual geometric structures used to model different types of Hamiltonian and Lagrangian systems fall within this framework, allowing for a unified and systematic mathematical treatment. In fact, all the examples mentioned earlier can be seen as doublets with class equal to the dimension of the manifold, where the class is given by the codimension of $\Ker\tau \cap \Ker\omega$. When the class is equal to the dimension of the manifold, the 2-covariant tensor field $B= \omega + \tau \otimes \tau$ is non-degenerate. 

In most of these contexts, the existence of a suitable bracket of functions is essential. The classical example is the Poisson bracket, which on a manifold $M$ is an $\R$-bilinear operation
\[
\{\cdot, \cdot\} \colon \Cinfty(M) \times \Cinfty(M) \longrightarrow \Cinfty(M)
\]
that is skew-symmetric and satisfies the Leibniz and Jacobi identities. Such a structure naturally arises in symplectic and cosymplectic geometry. A more general notion is that of a Jacobi bracket, which is the geometric framework encompassing locally conformally symplectic, contact, and cocontact structures. A Jacobi manifold is simply a manifold $M$ equipped with a bracket of functions $\{\cdot, \cdot\}$ that is skew-symmetric, local (usually referred to as satisfying a weak Leibniz identity), and fulfils the Jacobi identity. Namely, it defines a local Lie algebra on $\Cinfty(M)$ in the sense of Kirillov \cite{kir1976russ.math.surv.}.

A result due to Lichnerowicz \cite{Lichnerowicz1977LesVD} states that a Poisson structure on a manifold $M$ is equivalently defined by a bivector field $\Lambda \in \mathfrak{X}^2(M)$ satisfying $[\Lambda, \Lambda] = 0$, where $[\cdot, \cdot]$ denotes the Schouten--Nijenhuis bracket \cite{Mar_97,Tul_74}. Similarly, a Jacobi structure is equivalently given~\cite{Lichnerowicz1978LesVDJacobi} by a pair $(\Lambda, E)$, where $\Lambda \in \mathfrak{X}^2(M)$ and $E \in \mathfrak{X}(M)$ are such that $[\Lambda, \Lambda] =- 2 E \wedge \Lambda$ and $[\Lambda,E] = 0$. Associated with most of the aforementioned geometric structures, there exists a Poisson or Jacobi bivector field $\Lambda$ (along with a vector field $E$ in the Jacobi case), which is constructed from the defining differential forms.

More specifically, a unifying feature of symplectic, locally conformally symplectic, cosymplectic, contact, and cocontact geometry is the existence of a (not necessarily non-degenerate) $2$-form $\omega \in \Omega^2(M)$ that, together with the remaining structural data, induces a vector bundle isomorphism
\[
\flat \colon \T M \rightarrow \T^\ast M.
\]
With these two ingredients, one can always define the bivector field
\[
\Lambda(\alpha, \beta) \coloneqq -\omega (\flat^{-1}(\alpha), \flat^{-1}(\beta))\,,
\]
and then verify by direct computation whether it defines a Poisson or Jacobi structure. The computations involved get increasingly complicated as the geometric structure increases in generality. In fact, the calculation of the Schouten--Nijenhuis bracket is usually performed in Darboux coordinates, which significantly reduces the geometric understanding of the computation, and increases the amount of theory needed to conclude that the desired bracket is Poisson or Jacobi. Namely, one needs to prove the existence of adapted coordinates.

The objective of this paper is to shed light on these computations by using the isomorphism
\[
\flat \colon \T M \longrightarrow \T^\ast M
\]
as the core defining feature of the geometry. This approach naturally encompasses the aforementioned geometric structures within a single framework, while also allowing for more general situations. We identify this vector bundle isomorphism with the left contraction of a (generally neither symmetric nor skew-symmetric) $2$-covariant tensor field $B \in \Gamma \left( \T^\ast M \otimes \T^\ast M\right)$, and study the conditions under which the skew-symmetric part of its inverse is a Poisson or Jacobi bivector. The main contribution of this work is Theorem~\ref{thm:good-looking_formula}, which provides an intrinsic, coordinate-free formula for the Schouten--Nijenhuis bracket of the bivector field associated with~$B$. Although the formula is derived for a specific subclass of such tensor fields $B$, we show that one can always reduce the study to this case if the bivector has constant rank.

Specifically, if 
$B$ is a non-degenerate 2-covariant tensor such that $\T M = \Ker \omega \oplus \Ker S$, where $\omega$ and $S$ denote the skew-symmetric and symmetric parts of $B$, respectively, then for any $X,Y,Z \in \X (M)$ we have
    \begin{align*}
        \frac{1}{2}\llbracket\omega,\omega\rrbracket(X,Y,Z) =
        \d \omega ( X_\H, Y_\H, Z_\H) + S([X_\H,Y_\H],Z_\V) + S([Y_\H,Z_\H],X_\V) + S([Z_\H,X_\H],Y_\V)\,.
    \end{align*}
Here, $\llbracket \cdot,\cdot \rrbracket$ is a conjugate of the Schouten--Nijenhuis bracket by the bilinear form $B$, $\H = \Ker S$, $\V= \Ker \omega$, and $X_\H\coloneqq \pr_\H (X)$, $X_\V \coloneqq \pr_\V(X)$ denote the projections of the vector fields to the horizontal and vertical distributions, respectively,
as it will be explained later on.

This single formula allows us to systematically recover the classical Poisson and Jacobi brackets of all standard geometries (symplectic, locally conformally symplectic, cosymplectic, contact, and cocontact). Furthermore, we provide two novel applications of this framework, characterizing precisely when almost cosymplectic structures of order $p$ and fat bundles give rise to Poisson and Jacobi structures.

The paper is organised as follows.
In Section~\ref{section:multilinear_algebra}, we review essential aspects of multilinear algebra, focusing on the properties of non-degenerate 2-covariant tensors, their associated asymmetry endomorphism, and a generalised notion of orthogonality. 
In Section~\ref{Section:Conjugate_bracket}, we introduce the conjugate Schouten--Nijenhuis bracket and study its properties.
In Section~\ref{Section:One_tensor} we derive our main result: an intrinsic, coordinate-free formula characterizing the obstruction for a non-degenerate 2-covariant tensor to induce a Poisson or Jacobi structure. 
Section~\ref{section:Examples} is devoted to illustrating this framework across various geometric structures; we recover the classical brackets of symplectic, cosymplectic, contact, locally conformally symplectic, and cocontact geometries, and present new applications to almost cosymplectic manifolds of order $p$ and fat principal bundles. 
Finally, Section~\ref{section:conclusions} offers some concluding remarks and future outlooks. 

Unless otherwise stated, 
all vector spaces are assumed to be real and finite-dimensional.
Throughout the paper all manifolds are assumed to be smooth and paracompact, and all maps between them to be smooth. 

\section{Preliminaries on multilinear algebra}
\label{section:multilinear_algebra}
 
This section collects the multilinear algebra results used throughout the paper.
We study the algebraic properties of (not necessarily symmetric or skew-symmetric) non-degenerate 2-covariant tensors, their associated asymmetry endomorphism~(see Definition~\ref{def:asymmetry}), generalised notions of orthogonality, and the decomposition induced on the exterior algebra by a direct sum $E = {H} \oplus {V}$.

\subsection{Bilinear forms}

Let $E$ be a (finite-dimensional) real vector space.
Giving a bilinear form 
$B \colon E \times E \to \R$
on~$E$ amounts to giving a 2-covariant tensor
$B \in E^\ast \otimes E^\ast$. 
In this space we have a canonical involution 
\[
\mathsf{t} \colon E^\ast \otimes E^\ast \longrightarrow E^\ast \otimes E^\ast
\]
given by $\transp{(\alpha_1 \otimes \alpha_2)} \coloneqq \alpha_2 \otimes \alpha_1$ on decomposable elements, and extended by linearity.
It is well known 
that the decomposition of $E^\ast \otimes E^\ast$ into the $\pm1$ eigenspaces of this involution is the usual splitting in symmetric and skew-symmetric tensors
$$
E^\ast \otimes E^\ast =  
\sksym{2} (E^\ast) \oplus \sym{2}(E^\ast)\,.
$$
Hence, a bilinear form $B \in E^\ast \otimes E^\ast$ decomposes canonically into its skew-symmetric and symmetric parts $B = \omega + S$, where $\omega \in \sksym{2}(E^\ast)$ and $S \in \sym{2}(E^\ast)$. Explicitly, 
\[
\omega = \frac{B - \transp{B}}{2}\,, \qquad S = \frac{B + \transp{B}}{2}\,.
\]
The bilinear form $B$ naturally induces a linear map 
$$
\widehat{B} \colon E \rightarrow E^\ast
\,,
\qquad 
\widehat{B}(v_1) = B(v_1, \cdot)
\,.
$$
A similar linear map can be obtained from $B$ via right contraction, $v_2 \mapsto B(\cdot,v_2)$.
This other map is just $\widehat{\transp B}$.
In general $\widehat B$ and $\widehat{\transp B}$ are different,
since $B$ is \emph{not} assumed to be symmetric.

On the other hand,
the linear map $\widehat B$ has a transpose (or dual map)
$\transp{\widehat B} \colon E^{**} \cong E \to E^*$.
This coincides with $\widehat{\transp B}$,
therefore there is no danger of confusion.
Here we have made use of the canonical isomorphism 
$E \cong E^{**}$,
that holds only when $E$ is finite-dimensional.






The same applies to the extensions of $\widehat B$ to 
tensor powers and exterior powers of~$E$.
For instance,
\begin{lemma}
\label{lemma:evaluation_exterior_algebra}
Under the canonical identification of 
$\sksym{k}E^*$ with $(\sksym{k}E)^*$,
we have  
$$
\transp{(\sksym{k} \widehat{B})}
= 
\sksym{k} \transp{\widehat{B}}
\,.
$$
\end{lemma}


\subsection{Non-degenerate bilinear forms}

Henceforth we restrict to the non-degenerate case.
Recall that a bilinear form $B$ on~$E$ is \emph{non-degenerate} 
when the associated linear map
$\widehat{B} \colon E \to E^*$ is an isomorphism. 
This holds if, and only if,  $\transp{\widehat{B}}$ also is.
Using this isomorphism,
the spaces $E$ and $E^*$ can be swapped in any construction
involving them.
In particular, this isomorphism can be applied 
to the very bilinear form, or to its transpose.

\begin{definition}
The \dfn{inverse bilinear map} of~$B$ is the 
bilinear form 
$B^{-1} \colon E^* \times E^* \to \R$ such that 
$\widehat{B^{-1}}$ coincides with $\widehat{B}^{-1}$.
\end{definition}

\begin{remark}
Let $\operatorname{M}(B)$ be the matrix of the bilinear form~$B$.
Then, with the usual conventions of matrix calculus,
$\operatorname{M}(\widehat B) = \transp{\operatorname{M}(B)}$.
If $B$ is non-degenerate,
$\operatorname{M}(B^{-1}) = \operatorname{M}(B)^{-1}$.
\end{remark}

\begin{proposition}
\label{prop:description_of_inverse}
If $B$ is a non-degenerate bilinear form and $B^{-1}$ its inverse bilinear map, then:
\begin{equation}
(\widehat{B} \otimes \widehat{B}) \left(B^{-1}\right) 
= 
\transp B
\,,
\qquad
(\transp{\widehat{B}} \otimes \transp{\widehat{B}}) \left(B^{-1}\right) =\transp{ B}\,,
\end{equation}
\begin{equation}
(\widehat{B}^{-1} \otimes \widehat{B}^{-1})(\transp{B}) =B^{-1} \,,
\qquad
(\transp{\widehat{B}}^{-1} \otimes \transp{\widehat{B}}^{-1})
(\transp{B})=B^{-1} 
\,.
\end{equation}
\end{proposition}
\begin{proof}
We will just prove one expression, the others can be proven analogously. For any $v,w \in E$, we have
\begin{align*}
    (\widehat{B} \otimes \widehat{B}) (B^{-1}) (v,w) 
    &= B^{-1}(\widehat B(v), \widehat B(w)) 
    = \left\langle \widehat B(w), \widehat{B}^{-1}(\widehat B(v))\right\rangle \\
    &= \left\langle \widehat B(w), v\right\rangle 
    = B(w,v) 
    = \transp{B}(v,w)\,.
\end{align*}
The two expressions from the lower row can be proven by applying $(\widehat B^{-1}\otimes \widehat B^{-1})$ or $(\transp{\widehat B}^{-1}\otimes \transp{\widehat B}^{-1})$ to the ones above.
\end{proof}

To simplify notation, if $T$ is any 2-covariant (or 2-contravariant) tensor, we will usually write $\Ker T$ instead of $\Ker \widehat T$ to denote the kernel of the associated linear map.

\begin{remark}
Symplectic vector spaces and inner products are two extremal cases of non-degenerate $2$-covariant tensors, corresponding to the cases where $B$ is purely skew-symmetric and purely symmetric (and positive-definite), respectively. 
\end{remark}


\begin{proposition}
\label{prop:inner_product_non_deg}
Let $B = S + \omega$
be the decomposition of a bilinear form into its
symmetric and skew-symmetric components. We have:
\begin{enumerate}[\rm (i)]
\item 
$
\Ker S \cap \Ker \omega \subset \Ker B
$.   
\item 
The equality 
$
\Ker S \cap \Ker \omega = \Ker B
$
holds if $S$ is (positive or negative) semidefinite.
\item 
If $S$ is indefinite, 
one can find an $\omega$ such that the inclusion is strict.
\item If $E=\Ker \omega \oplus \Ker S$, then $B$ is non-degenerate.
\end{enumerate}
\end{proposition}
\begin{proof}
The first statement is obvious.
If $v \in E$ and $\iota_v (\omega + S) = 0$, we have in particular that 
$S(v, v) = S(v,v) + \omega(v,v) = 0$.
Since $S$ is semidefinite and $\omega$ skew-symmetric, $v \in \Ker S$, and
hence $\iota_v \omega = 0$.
The third statement is easily proved by expressing $S$ in an appropriate basis. 
For the fourth statement, let $u \in \Ker B$ and decompose it as $u = v + h$ with $v \in \Ker \omega$ and $h \in \Ker S$. 
Since $E = \Ker \omega \oplus \Ker S$, we have $\Ker \omega \cap \Ker S = 0$, so the restrictions $\omega|_{\Ker S}$ and $S|_{\Ker \omega}$ are non-degenerate. 
For any $h' \in \Ker S$, we have that
$$
0 = B(v+h, h') = \omega(v+h, h') + S(v+h, h') = \omega(h, h')+S(h',v+h)=\omega(h, h')\,,
$$
which implies that $h = 0$, by the non-degeneracy of $\omega|_{\Ker S}$. Analogously, testing against any $v' \in \Ker \omega$ yields $S(v, v') = 0$, which implies $v = 0$. Thus $u = 0$, and so $\Ker B = 0$.
\end{proof}

\begin{corollary}
\label{cor:inner_product_non_deg}
If $S$ is an inner product, 
then $B=S+\omega$ is non-degenerate, for any skew-symmetric 2-covariant tensor $\omega$.  
\end{corollary}

\begin{proposition} Let $B \in E^\ast \otimes E^\ast$ be non-degenerate, and let $\omega$ denote its skew-symmetric part. Then, the skew-symmetric part of $B^{-1}$ is
$$\Lambda =-\left(\sksym{2} \widehat B^{-1}\right)(\omega)= -\left(\sksym{2} \transp{\widehat B}^{-1}\right) (\omega)\,.$$
\end{proposition}
\begin{proof} It follows from Proposition  \ref{prop:description_of_inverse}, since for every $\alpha, \beta \in E^\ast$ we have
\begin{align*}
    \left( \frac{B^{-1}  - \transp{B}^{-1}}{2}\right)(\alpha, \beta) &= \frac{B^{-1}(\alpha, \beta) - B^{-1}(\beta, \alpha)}{2}\\
    &= \frac{\transp{B}({{\widehat{B}}}^{-1}(\alpha),{{\widehat{B}}^{-1}(\beta)}) - \transp{B}({\widehat{B}}^{-1}(\beta),{\widehat{B}}^{-1}(\alpha)) }{2}\\
    &= -\omega({\widehat{B}}^{-1}(\alpha), {\widehat{B}}^{-1}(\beta) ) = -\left(\left( \sksym{2} {\widehat B}^{-1}\right) (\omega)\right)(\alpha, \beta)\,,
\end{align*}
and analogously for the other expression.
\end{proof}
As a direct consequence, the skew-symmetric part of $\transp{B}^{-1}$ is $$-\Lambda = \left(\sksym{2} \widehat B^{-1}\right)(\omega)=\left( \sksym{2} \transp{\widehat B}^{-1}\right) (\omega)\,.$$

\begin{remark}
    For any $\alpha_1,\alpha_2 \in E^*$ we have
    $$\Lambda(\alpha_1,\alpha_2) = -\omega(\widehat B^{-1}(\alpha_1),\widehat B^{-1}(\alpha_2))=-\omega(\transp{}\widehat B^{-1}(\alpha_1),\transp{}\widehat B^{-1}(\alpha_2))\,.$$ From this, it follows that the linear map associated to the bivector $\Lambda$ is given by
     $$\widehat \Lambda = -\transp{ \widehat B}^{-1} \circ \widehat \omega\circ  \widehat B^{-1}=-{ \widehat B}^{-1} \circ \widehat \omega\circ  \transp{}\widehat B^{-1}\,,$$ and hence, $\alpha \in \Ker \Lambda $ if, and only if, $\alpha \in \widehat B(\Ker \omega)= \widehat S(\Ker\omega)$.

\end{remark}

We have analogous results for the symmetric part of $B^{-1}$.
\begin{proposition} Let $B \in E^\ast \otimes E^\ast$ be non-degenerate, and let $S$ denote its symmetric part. Then, the symmetric part of ${B}^{-1}$ is 
$$\Sigma = \left(\textstyle\sym{2} {\widehat B}^{-1}\right) (S)=\left(\textstyle\sym{2} \transp{\widehat B}^{-1}\right) (S)\,.$$
\end{proposition}
\begin{proof} The proof is analogous to that presented for the skew-symmetric part.
\end{proof}
As a direct consequence, the symmetric part of $\transp{B}^{-1}$ is also $\Sigma$.

\begin{remark}
    For any $\alpha_1, \alpha_2 \in E^*$ we have
    $$\Sigma(\alpha_1,\alpha_2) = S(\widehat B^{-1}(\alpha_1),\widehat B^{-1}(\alpha_2))=S(\transp{}\widehat B^{-1}(\alpha_1),\transp{}\widehat B^{-1}(\alpha_2))\,.$$ From this,
    $$\widehat \Sigma = \transp{\widehat B}^{-1} \circ \widehat S \circ \widehat B^{-1}=\widehat B^{-1}\circ \widehat S \circ \transp{\widehat B}^{-1}\,,$$
    and so 
    $\alpha \in \Ker \Sigma$ if, and only if, $\alpha \in \widehat \omega(\Ker S)$.
\end{remark}

\subsection{Asymmetry and orthogonality}

The classification of skew-symmetric and symmetric bilinear forms is well known. 
There is a less known classification of all the orbits of the 
$\mathrm{GL}(E)$-action on $E^\ast \otimes E^\ast$. 
This classification is done via the \emph{asymmetry} isomorphism
\cite{cw2001journalofpureandappliedalgebra, rie1974journalofalgebra}. 
\begin{definition}\label{def:asymmetry}
Let $B$ be a \emph{non-degenerate} bilinear form on~$E$.
The \dfn{asymmetry} of~$B$ is the linear isomorphism
$$
J \coloneqq 
\transp{\widehat{B}}^{-1} \circ \widehat B \colon 
E \longto E
\,.
$$
Its inverse map is 
$J^{-1} = \widehat B^{-1} \circ \transp{\widehat{B}}$.
\end{definition}

By definition, 
$$
B(u,J(v)) = B(v,u)\,,
$$
for all $u,v\in E$, and so
$$
B(J(u),J(v)) = B(u,v)
\,.
$$
Also,
$$\widehat{\omega} = \frac{1}{2}(\widehat{B} - \transp{\widehat{B}}) = \frac{1}{2}\transp{\widehat{B}}(J - \Id)\,,\qquad \text{and} \qquad\widehat{S} = \frac{1}{2}(\widehat{B} + \transp{\widehat{B}}) = \frac{1}{2}\transp{\widehat{B}}(J + \Id)\,.$$
 Hence, $J=\Id$ if, and only if, $B$ is symmetric, and $J=-\Id$ if, and only if, $B$ is skew-symmetric.
More specifically,
    $$v\in \Ker \omega \iff J(v) = v\,, \qquad \text{and}\qquad
        v\in \Ker S \iff J(v) = -v\,.
    $$
In particular, this result implies that $\Ker\omega$ and $\Ker S$ are always in direct sum when $B$ is non-degenerate. 
Also, $E=\Ker \omega\oplus \Ker S$ if, and only if, $E$ is a direct sum of the $\pm 1$-eigenspaces of $J$, which is equivalent to $J$ being an involution, i.e.\ $J^2=\Id$. 

\begin{remark}
When $E$ is even-dimensional and $(\omega, g)$ defines a Kähler structure on~$E$, 
the asymmetry of $B = g + \omega$ corresponds precisely to the induced complex structure~$J$
(under the appropriate sign convention).
\end{remark}

Similar results follow for the transpose of $J$, which is
$$
\transp{J}= \transp{\widehat B}\circ \widehat B^{-1}\colon E^* \longto E^*\,.
$$
We have
$$\widehat{\Lambda} = \frac{1}{2}({\widehat{B}^{-1}}-\transp{\widehat{B}^{-1}}) = \frac{1}{2}\transp{\widehat{B}}^{-1}( \transp{J}- \Id)\,, \qquad\widehat{\Sigma} = \frac{1}{2}(\transp{\widehat{B}^{-1}}+\widehat{B}^{-1}) = \frac{1}{2}\transp{\widehat{B}}^{-1}(\transp{J}+\Id )\,,$$
so that
$$ 
\alpha\in \Ker  \Lambda \Longleftrightarrow \transp{J}(\alpha) =\alpha\,,
\quad \text{and} \quad 
\alpha\in \Ker  \Sigma \Longleftrightarrow \transp{J}(\alpha) =-\alpha\,.
$$  
And thus, both kernels are in direct sum and $E^* = \Ker  \Lambda\oplus \Ker  \Sigma$ if, and only if, $\transp{J}^2= \Id$.
As $J^2=\Id$ is equivalent to $\transp{J}^2= \Id$, we have
\begin{equation}\label{eq:direct_sum_iff_dual}
    E=\Ker \omega \oplus \Ker S\Longleftrightarrow E^* = \Ker  \Lambda\oplus \Ker  \Sigma\,.
\end{equation}
This case will be of particular importance in later sections.

The notion of \emph{orthogonality} is fundamental in symplectic and Riemannian geometry, as well as in the study of Poisson and Jacobi manifolds 
(see \cite{Vai_94} for the Poisson case, and \cite{ILM_97, Vit_2018} for the Jacobi case). 
We now define and study the general notion of orthogonality for a 2-covariant tensor. 
It should be noted that these notions of orthogonality with respect to a not necessarily symmetric or skew-symmetric non-degenerate 2-covariant tensor have already been studied and used in several works (see, for instance,~\cite{LL_19a, jl2025journalofgeometryandphysics}), we provide a detailed account here for completeness.

\begin{definition} Let $T$ be a bilinear form on $E$ and $W \subseteq E$ be a subspace. We define its \dfn{$T$-orthogonal complement} as 
\[
W^{\perp, {T}} = 
\{v \in E \mid 
T(v, w) = 0\ \forall\, w \in W\} = 
\widehat{T}^{-1} \left( W^\circ \right)\,, 
\]
where $W^\circ \subseteq E^\ast$ stands for the annihilator of $W$ and $\widehat{T}^{-1}$ denotes the preimage under $\widehat{T}$. 
\end{definition}

\begin{remark}
    As, in general, a bilinear form is not symmetric nor skew-symmetric,
    we can define another notion of orthogonality by contracting on the right. This coincides with the orthogonal with respect to the transpose $\transp{\,T}$. That is, the \dfn{$\transp{\,T}$-orthogonal complement} of $W\subseteq E$ is 
$$
W^{\perp, \transp{\,T}} = \{v \in E \ \mid\ T(w, v) = 0\,,\ \forall\, w \in W\} = \transp{\widehat{T}}^{-1}\left( W^\circ \right)\,. 
$$
Notice that, by the definition of annihilator, we also have
$$W^{\perp,T} = \left(\transp{\widehat T}(W)\right)^{\circ}\qquad \text{and} \qquad W^{\perp,\transp{\,T}}=\left(\widehat T(W)\right)^{\circ}\,.$$
\end{remark}

For the rest of the section we will deal with the orthogonal complements defined by a non-degenerate bilinear form, that we denote by $B$.
To simplify notation we assume $B$ fixed and write
\[
W^{\perp} = W^{\perp, B} \qquad \text{and} \qquad \leftperp{W} = W^{\perp, \transp{B}}\,.
\]
Observe that, using that $B$ is non-degenerate and that $\dim W^\circ = \codim W$, we have
\begin{equation}\label{eq:dimensions_orth}
    \dim {}^{\perp}W=\dim W^{\perp} = \codim W\,,
\end{equation}
from which it follows that $\dim E= \dim W+\dim W^\perp=\dim W+\dim \leftperp W$.

We study the relation between the two notions of orthogonality in the following proposition.

\begin{proposition}
\label{prop:relation_orthogonals} For every $W \subseteq E$, we have the following: 
\begin{enumerate}[\rm (i)]
    \item${W}^\perp = \leftperp (J^{-1}(W)) = J^{-1}(\leftperp W)$ and $\leftperp{W} =J(W)^{\perp} = J(W^{\perp})$, where $J = \transp{\widehat{B}}^{-1} \circ \widehat{B}$ is the asymmetry of $B$.
    \item $\leftperp{(W^{\perp})} = W$ and $ (\leftperp{W})^{\perp} = W$.
    \item $W^{\perp \perp} = J^{-1}(W)$, and $\leftperp{\leftperp{W}} = J(W)$.
\end{enumerate}
\end{proposition}

\begin{proof} \leavevmode
\begin{enumerate}[\rm (i)]
    
    \item Indeed, both equalities follow from the defining equality of $J$, which can be written as
    \[
    B(u, J(v)) = (\transp{B})(u, v)\,.  
    \]
    \item It is sufficient to prove the first equality. As the dimensions match, it is enough to show that $W \subseteq {}^{\perp}(W^{\perp})$. By definition, if $v \in W^{\perp}$ we have $B(v, w) = 0$, for all $w \in W$ and, in particular, this implies that $w \in {}^{\perp}(W^{\perp})$.
    \item This follows directly from the last two results. 
\end{enumerate}
\end{proof}

Next, we study the properties of these orthogonal complements with respect to the intersection and the sum of vector subspaces.

\begin{proposition}
\label{prop:2.14}
For all subspaces $W, W_1,W_2\subseteq E$, the following identities hold
\begin{enumerate}[(i)]
    \item $(W_1 \cap W_2)^{\perp} = W_1^{\perp} + W_2^{\perp}$ and $(W_1+ W_2)^{\perp} = W_1^{\perp} \cap W_2^{\perp}$. The same identities hold for the left orthogonal.
    \item If $\omega$ and $S$ denote the skew-symmetric and symmetric part of $B$, we have $W^{\perp} \cap \leftperp{W} = W^{\perp, \omega} \cap W^{\perp, S}$. In particular, $W^{\perp} \cap \leftperp{W} \cap \Ker S = W^{\perp, \omega} \cap \Ker S$.
    \item $(W^{\perp} +\leftperp{W})^{\perp} = J^{-1}(W) \cap W$ and $\leftperp{(W^{\perp} + {\leftperp{W}})}= J(W) \cap W$.
    \item  $B|_{W}$ is non-degenerate if, and only if, $E = W \oplus W^{\perp}$.
\end{enumerate}
    
\end{proposition}
\begin{proof}\leavevmode
    \begin{enumerate}[(i)]
       \item We have
       $$(W_1\cap W_2)^\perp = \widehat B^{-1}((W_1\cap W_2)^{\circ})=\widehat B^{-1}(W_1^{\circ}+W_2^{\circ})=\widehat B^{-1}(W_1^{\circ})+\widehat B^{-1}(W_2^{\circ})\,,$$
       where we have used the properties of the annihilator and that $\widehat B$ is an isomorphism. The proofs of the other equalities are analogous.
        \item Indeed, $v \in W^{\perp} \cap {}^{\perp}W$ if, and only if, we have 
    \[
    \omega(v, w) + S(v, w) = 0\,, \qquad \text{and} \qquad -\omega(v, w) + S(v, w) = 0\,,
    \]
    for every $w \in W$. Summing and subtracting both equalities we get $\omega(v, w) = S(v, w) = 0$, for all $w \in W$.
    \item By the first property and Proposition \ref{prop:relation_orthogonals}, we have
    \[
    (W^{\perp} + {}^{\perp}W)^{\perp} = (W^{\perp})^{\perp} \cap ({}^{\perp}W)^{\perp} = J^{-1}(W) \cap W\,,
    \]
    the proof of the other equality is analogous.
     \item It follows from observing that $B|_W$ is non-degenerate if, and only if, $W \cap W^{\perp} = \{0\}$ and using Equation~\eqref{eq:dimensions_orth}.
    \end{enumerate}
\end{proof}

\subsection{Direct sum decomposition of the space}
\label{sec:decomposition}
Let $B$ be a non-degenerate 2-covariant tensor, and $\omega,S$ be, respectively, its skew-symmetric and symmetric parts. 
The case where $E=H\oplus V$, with $H\coloneqq \Ker S$ and $V \coloneqq \Ker \omega$, will be of particular relevance to our study. Recall that if we have this direct sum decomposition condition then, necessarily, the tensor $B$ has to be non-degenerate (by Proposition~\ref{prop:inner_product_non_deg}). 

In this case, we can decompose every $w \in E$ uniquely as $w= w_h + w_v$, with $w_h \in H$ and $w_v \in V$, and so we have well-defined projections $\pr_{H}\colon E\to H\subseteq E$ and $\pr_V\colon E \to V\subseteq E$.
Note that, as $B=\omega +S$ we have $\widehat B\circ \pr_H= \widehat\omega$ and $\widehat B\circ{\pr_{V}} =\widehat S$, so
$$\widehat B^{-1}\circ \widehat\omega =\pr_H\,, \qquad \widehat B^{-1}\circ \widehat S = \pr_V\,,$$
and, similarly, as $\transp{B} = -\omega +S$, we also have $\transp{\widehat B}^{-1}\circ \widehat\omega =-\pr_H$ and $ \transp{\widehat{B}}^{-1}\circ \widehat{S} = \pr_V$.
Hence, using that $\widehat \Lambda= -\transp{\widehat B}^{-1}\circ \widehat\omega \circ \widehat B^{-1}$ we have
\begin{equation}\label{eq:Lambda_decomposition}
    \widehat \Lambda= \pr_H \circ \widehat B^{-1}=-\pr_H \circ \transp{}\widehat B^{-1}\,.
\end{equation}
Also,
$$H^{\perp} = {}^{\perp}{H} = {V} \qquad \text{and}\qquad {V}^{\perp} = {}^{\perp}{V} = {H}\,,$$
with the orthogonal complements taken with respect to $B$, as in the last section.

Although we do not have a direct sum decomposition of the whole space for every non-degenerate tensor $B$ (for example, $\Ker\omega\oplus \Ker S=\{0\}$ whenever both $\omega$ and $S$ are non-degenerate), if we are just studying the bivector $\Lambda$, as we do in the sequel, then we may always reduce to this case. Indeed, we have the following result:
\begin{proposition}
\label{prop:existence_of_equivalent_B}
    Let $B$ be a non-degenerate 2-covariant tensor. Then, there exists a 2-covariant tensor $B'$ such that
    \begin{enumerate}[\rm (i)]
        \item It is non-degenerate.
        \item If $\omega '$ and $S'$ denote the skew-symmetric and symmetric part of $B'$, respectively, we have $\Ker\omega ' \oplus \Ker S' = E$.
        \item The skew-symmetric parts of $B^ {-1}$ and $(B')^ {-1}$ coincide.
    \end{enumerate}
\end{proposition}
\begin{proof} Let $\Lambda$ denote the skew-symmetric part of ${B}^{-1}$ and choose a complementary subspace $F \subseteq E^\ast$ such that $E^\ast = \Ker\Lambda \oplus F$. 
Choose a symmetric 2-contravariant tensor $\Sigma \in \sym{2}(E)$ such that $\Ker\Sigma = F$. Then, the 2-contravariant tensor $\Lambda + \Sigma$ is non-degenerate, by Proposition~\ref{prop:inner_product_non_deg}, and, by Equation~\eqref{eq:direct_sum_iff_dual}, it is enough to consider 
\[
B' \coloneqq {\left(\Lambda + \Sigma \right)}^{-1}\,.
\] 
\end{proof}
Notice that every 2-covariant tensor $B'$ satisfying the properties of Proposition \ref{prop:existence_of_equivalent_B} can be built by choosing a decomposition $E^\ast = \Ker \Lambda \oplus F$ and a symmetric 2-contravariant tensor $\Sigma \in \sym{2}(E)$ satisfying $F = \Ker \Sigma$ (equivalently, that $\Sigma$ is non-degenerate on $\Ker \Lambda$). Hence, given a non-degenerate 2-covariant tensor $B$, the set of 2-covariant tensors $B'$ satisfying Proposition \ref{prop:existence_of_equivalent_B} is in bijection with
\begin{align*}
 \left\{\text{Non-degenerate } \Sigma \in \sym{2}((\Ker \Lambda)^*)\right\} \times \left\{\text{Direct sum decompositions } E^\ast = \Ker \Lambda \oplus F\right\}\,.   
\end{align*}

\begin{remark}
When the bivector $\Lambda$ is non-degenerate, there is a canonical choice of $B'$, which is just $B'=\Lambda^{-1}$ (i.e.\ taking $F=E^*$ and $\Sigma=0$).
If the symmetric part of $B$ is an inner product $g$, there is also a canonical choice of $F$ and~$\Sigma$: we can set $F \coloneqq (\Ker{\Lambda})^{\perp, g^{-1}}$ and $\Sigma \coloneqq g^{-1}|_{\Ker \Lambda}$.

Let us point out that, in general, $B'$ cannot be constructed by altering only $S$ while keeping~$\omega$ fixed. For example, consider
$$B = \begin{pmatrix} 1 & 1 \\ -1 & 1 \end{pmatrix} = \begin{pmatrix} 0 & 1 \\ -1 & 0 \end{pmatrix} + \begin{pmatrix} 1 & 0 \\ 0 & 1 \end{pmatrix} = \omega + S\,.$$
Since both $\omega$ and $S$ are non-degenerate, achieving $E = \Ker \omega \oplus \Ker S'$ requires $S' = 0$, leading to $B' = \omega$. 
However, computing the skew-symmetric parts of ${B}^{-1}$ and $B'^{-1}$, we obtain 
$$\Lambda = \frac{1}{2}\begin{pmatrix} 0 & -1 \\ 1 & 0 \end{pmatrix}\qquad  \text{and}\qquad \Lambda' = \begin{pmatrix} 0 & -1 \\ 1 & 0 \end{pmatrix}\,,$$
which do not coincide. 

\end{remark}

\begin{remark}
\label{remark:Jacobi_orthogonal}
Let $B$ be a non-degenerate $2$-covariant tensor on $E$ such that $E = \Ker \omega \oplus \Ker S$, where $\omega$ and $S$ denote its skew-symmetric and symmetric parts, respectively. Consider the bivector 
$\Lambda = -\big( \sksym{2} {\widehat B}^{-1}\big) (\omega)$. Given a subspace $W \subseteq E$ we define $W^{\perp, \Lambda} \coloneqq \widehat{\Lambda} \left( W^\circ\right)$. This orthogonal corresponds to the point-wise Poisson and Jacobi orthogonal \cite{deLeonIzquierdoLopez2024}. We can recover it in terms of $B$ as follows. In full generality, we have that 
\[
W^{\perp, \omega} \cap {H} \subseteq W^{\perp, \Lambda}\,.
\]
However, if $W$ is such that $W = (W \cap {H}) \oplus (W \cap {V})$, it is easy to show that the $\Lambda$-orthogonal can be written in terms of the $B$-orthogonal complements as
\[
W^{\perp, \Lambda} = W^{\perp, \omega} \cap {H} = W^{\perp} \cap \leftperp W \cap {H}\,.
\]

\end{remark}

Finally, let us recall that a direct sum decomposition ${H} \oplus{V}$ induces a canonical decomposition on the exterior algebra
$$\sksym{k} E^\ast =\bigoplus_{j = 0}^k \left( \sksym{j}{H}^\ast \otimes \sksym{k - j} {V}^\ast\right)\,.
$$
\begin{definition} An element $\alpha \in \sksym{k} E^\ast$ is said to be of \dfn{type} $(j , k-j)$ if it lies in $\sksym{j}{H}^\ast \otimes \sksym{k-j} {V}^\ast$. 
\end{definition}

In particular, any form $\alpha \in \sksym{k} E^\ast$  decomposes as a sum 
\[
\alpha = \alpha_{(k, 0)} + \cdots + \alpha_{(0,k)}\,,
\]
and each component $\alpha_{(j, k-j)}$ is completely determined by the values of the map
\[
\underbrace{H \times \cdots \times H}_{j \text{ times}} \times \underbrace{V \times\cdots \times V}_{k-j \text{ times}}  \longrightarrow \R\,, \qquad(h_1, \dots, h_j, v_1, \dots, v_{k-j}) \mapsto \alpha(h_1, \dots, h_j, v_1, \dots, v_{k-j})\,.
\]

\section{Brackets induced by a non-degenerate 2-covariant tensor}
\label{Section:Conjugate_bracket}

From this point onwards, we will study $2$-covariant tensor fields on a manifold $M$ which satisfy a non-degeneracy condition point-wise. More particularly, we will investigate the conditions under which a section
\[
B \colon M \rightarrow \T^\ast M \otimes \T^\ast M\,,
\]
such that $B|_x \in \T^\ast _x M \otimes \T^\ast_x M$ is a non-degenerate $2$-covariant tensor for every $x\in M$, defines a Poisson or Jacobi bracket. 
We will denote by 
$\transp{B} \in \Gamma(\T^\ast M\otimes \T^\ast M)$,
$B^{-1} \in \Gamma(\T M \otimes \T M)$,
and
$\widehat{B} \colon \T M \rightarrow \T^\ast M$ the tensor fields and the vector bundle isomorphism, respectively, defined point-wise from~$B$.
The symmetric and skew-symmetric parts of$B$ are defined in the same way.

\subsection{The conjugate Schouten--Nijenhuis bracket}

Let us assume that $B$ is a non-degenerate 2-covariant tensor field on $M$, and that $\omega$ and $S$ are, respectively, its skew-symmetric and symmetric parts.
We will denote by
\[
\sksym{k}\widehat B \colon \sksym{k} \T M \longrightarrow \sksym{k} \T ^\ast M \,,
\]
the linear maps associated to the natural extensions of $\widehat B$ to the exterior algebra. To study the conditions under which the associated bivector field $\Lambda = -\big( \sksym{2} {\widehat B}^{-1}\big) (\omega)$ can be completed to a Jacobi bracket, we need to compute $[\Lambda, \Lambda]$, where $[\cdot, \cdot]$ denotes the Schouten--Nijenhuis bracket~\cite{Mar_97}. Indeed, celebrated results due to Lichnerowicz \cite{Lichnerowicz1977LesVD, Lichnerowicz1978LesVDJacobi} state that a bivector field $\Lambda \in \mathfrak{X}^2(M)$ defines a Poisson structure or can be completed to a Jacobi structure if, and only if, 
\[
[\Lambda, \Lambda] = 0 \qquad \text{or} \qquad [\Lambda, \Lambda] =-2 E \wedge \Lambda\,\text{ and }\,[\Lambda,E]=0\,,
\]
for some $E \in \mathfrak{X}(M)$, respectively\footnote{The sign in front of $E\wedge\Lambda$ depends on the sign convention adopted for the Schouten--Nijenhuis bracket. Throughout this paper we use the convention of~\cite{Mar_97}.}. To compute this, it will be useful to define the following bracket, which is a conjugate of the usual Schouten--Nijenhuis bracket.

\begin{definition} Let $B$ be a non-degenerate 2-covariant tensor field. For any pair of forms $\alpha \in \Omega^a(M)$ and $\beta \in \Omega^b(M)$, the \emph{conjugate Schouten--Nijenhuis bracket associated with $B$}, $\llbracket\cdot, \cdot\rrbracket\colon \Omega^a(M) \times \Omega^b(M) \longrightarrow \Omega^{a+b-1}(M)$, is given by
\[
\llbracket\alpha, \beta\rrbracket \coloneqq \sksym{a+b-1}\transp{\widehat{B}} \left(\left[\sksym{a}\transp{\widehat{B}}^{-1}(\alpha), \sksym{b}\transp{\widehat{B}}^{-1}(\beta)\right] \right)\,.
\]
\end{definition}

With the conjugate bracket, the Poisson and Jacobi conditions of $[\Lambda, \Lambda] = 0$ or $[\Lambda,\Lambda]=-2 E\wedge \Lambda$ and $[\Lambda,E]=0$
for some $E\in \X(M)$, can be equivalently written as $$\llbracket\omega, \omega \rrbracket = 0\qquad \text{or} \qquad \llbracket\omega, \omega\rrbracket = 2 \alpha \wedge \omega\,\text{ and }\,\llbracket \omega,\alpha \rrbracket=0\,,$$
for some $\alpha \in \Omega^1(M)$, and then $E=\transp{\widehat B}^{-1}(\alpha)$.
Hence, the problem is reduced to computing $\llbracket \omega, \omega \rrbracket$. While this computation is generally rather cumbersome, if we have that $\T M = \Ker\omega \oplus \Ker S$ (with $\widehat \omega$ and $\widehat S$ of constant rank), we can give the explicit formula:
    \begin{align*}
        \frac{1}{2}\llbracket\omega,\omega\rrbracket(X,Y,Z) =
        \d \omega ( X_\H, Y_\H, Z_\H) + S([X_\H,Y_\H],Z_\V) + S([Y_\H,Z_\H],X_\V) + S([Z_\H,X_\H],Y_\V)\,,
    \end{align*}
    for any $X,Y,Z \in \X(M)$.
Here, $\H = \Ker S$, $\V= \Ker \omega$, and $X_\H\coloneqq \pr_\H (X)$, $X_\V \coloneqq \pr_\V(X)$ denote the projections of the vector fields to the horizontal and vertical distributions, respectively.
This will be proved in the next section.

\subsection{Properties of the conjugate Schouten--Nijenhuis bracket}

Here we give some elementary properties of the conjugate Schouten--Nijenhuis bracket $\llbracket\cdot, \cdot \rrbracket$ associated with a non-degenerate 2-covariant tensor $B$.
\begin{proposition}
\label{prop:properties_bracket}
The bracket $\llbracket\cdot, \cdot \rrbracket$ satisfies the following properties:
\begin{enumerate}[\rm(i)]
    \item It is graded skew-symmetric: For every $\alpha \in \Omega^a(M)$ and $\beta \in \Omega^b(M)$ we have $\llbracket\alpha, \beta \rrbracket = -(-1)^{(a-1)(b-1)} \llbracket\beta, \alpha \rrbracket$.
    \item It satisfies the graded Jacobi identity: For every $\alpha \in \Omega^a(M)$, $\beta \in \Omega^b(M)$ and $\gamma \in \Omega^c(M)$, we have
    \[
    (-1)^{(a-1)(c-1)} \llbracket\alpha ,\llbracket\beta, \gamma \rrbracket \, \rrbracket + \text{cyclic permutations} = 0\,. 
    \]
    \item It satisfies a graded Leibniz identity: For every $\alpha \in \Omega^a(M)$, $\beta \in \Omega^b(M)$ and $\gamma \in \Omega^c(M)$, we have
    \[
    \llbracket\alpha, \beta \wedge \gamma\rrbracket = \llbracket\alpha, \beta \rrbracket \wedge \gamma + (-1)^{(a-1)b} \beta \wedge \llbracket\alpha, \gamma \rrbracket\,.
\]
    \item If $f \in \Cinfty(M)$, $\alpha \in \Omega^a(M)$ and $\beta \in \Omega^b(M)$ and we define $X_f \coloneqq {\widehat{B}}^{-1}(\d f)$, then
    \[
    \inn{X_f} \llbracket\alpha, \beta \rrbracket = \llbracket \inn{X_f}\alpha,\beta \rrbracket + (-1)^{a-1} \llbracket\alpha, \inn{X_f}\beta \rrbracket\,.
    \]
\end{enumerate}
\end{proposition}

\begin{proof} 
The first three properties follow from the properties of the Schouten--Nijenhuis bracket.
Let us prove~{\rm (iv)}, as it does not immediately follow from algebraic considerations. Recall (see~\cite[Prop.\ 4.3]{Mar_97}),
that for every closed $1$-form $\varepsilon \in \Omega^1(M)$ and multivectors $U \in \mathfrak{X}^{p}(M)$ and $V \in \mathfrak{X}^{q}(M)$ on $M$, we have 
\begin{equation}\label{eq:multivectors_formula_remark}
\inn{\varepsilon}[U, V] =[\inn{\varepsilon}U,  V]+ (-1)^{(p-1)}[ U,\inn{\varepsilon} V]  \,.    
\end{equation}
In particular, for arbitrary $\alpha \in \Omega^a(M)$, $\beta \in \Omega^b(M)$ and $U \in \mathfrak{X}^{a+b - 2}(M)$ we have 
\begin{align*}
    \langle \inn{X_f}\llbracket\alpha, \beta \rrbracket, U \rangle =  \left\langle \left[\sksym{a}\transp{\widehat{B}}^{-1}(\alpha), \sksym{b}\transp{\widehat{B}}^{-1}(\beta) \right], \d f \wedge  \sksym{a+ b - 2} \widehat B(U)\right\rangle\,.
\end{align*}
The equality follows from~\eqref{eq:multivectors_formula_remark} and the fact that $\inn{\d f}\left( \sksym{a} \transp{\widehat B}^{-1}\right) (\alpha) = \sksym{a-1} \transp{\widehat B}^{-1} (\inn{X_f}\alpha)$.
\end{proof}

\begin{remark} For the sake of completeness, we would like to mention that in \cite[Prop.\ 4.3]{Mar_97} the author writes that, for any multivector fields $P\in \X^p(M)$, $Q\in \X^q(M)$ and closed 1-form $\varepsilon\in \Omega^1(M)$, one has
\[
j_\varepsilon [P, Q] = [P, j_\varepsilon Q] + (-1)^{q-1}[j_\varepsilon P, Q]\,,
\]
where $j_\varepsilon$ denotes the contraction on the right, rather than $\inn{\varepsilon}$ which denotes contraction on the left. Notice that both are related by $j_\varepsilon P = (-1)^{p-1} \inn{\varepsilon} P$. Hence the above formula can be written as 
\[
(-1)^{p+q} \inn{\varepsilon}[P, Q] = (-1)^{(q-1)} [P, \inn{\varepsilon} Q] + (-1)^{(q-1)+(p-1)} [\inn{\varepsilon} P, Q]\,,
\]
which recovers Equation~\eqref{eq:multivectors_formula_remark}.
\end{remark}

\begin{remark}
Note that by Proposition~\ref{prop:properties_bracket}, if $X_f = {\widehat B}^{-1}(\d f) \in \X(M)$ and $\omega \in \Omega^2(M)$, then
\begin{equation}\label{eq:bracket_omega_Xf}
\inn{X_f} \llbracket\omega, \omega\rrbracket = \llbracket\iota_{X_f}\omega, \omega\rrbracket - \llbracket\omega, \inn{X_f}\omega\rrbracket = -2 \llbracket\omega, \inn{X_f}\omega\rrbracket\,.
\end{equation}

\end{remark}

Let $E$ be a vector field and $T$ a tensor field. We denote by $\Lie_E \widehat T \coloneqq \widehat {\Lie_E T}$ the linear map induced by the contraction of $\Lie_E T$ on its first argument.

It will be useful to derive another expression for the bracket above, when one of the arguments is a 1-form. 
First, we need to observe the following:
\begin{lemma}
\label{lemma:Lie_derivative_of_inverse}
Let $B$ be a non-degenerate 2-covariant tensor and $B^{-1}$ the 2-contravariant tensor associated with its inverse. Then, for any $E\in \X(M)$,
$$ \Lie_E  \sksym{k} \widehat{B}^{-1} = - \sksym{k}\widehat{B}^{-1} \circ\left( \Lie_E \sksym{k} \widehat{B}\right) \circ \sksym{k}\widehat{B}^{-1} \,.$$
The same identity holds also for the transpose maps:
\[
\left(\Lie_{E} \sksym{k}\transp{\widehat{B}}^{-1}\right) = - \sksym{k}\transp{\widehat{B}}^{-1} \circ \left({\Lie_E \sksym{k} \transp{\widehat B}}\right)\circ\sksym{k}\transp{\widehat B}^{-1} \,.
\]

\end{lemma}
\begin{proof}
Indeed, from the equality $\sksym{k}\widehat{B}^{-1} \circ \sksym{k}\widehat{B} = \mathrm{Id}_{\sksym{k}\T M}$, the fact that the Lie derivative commutes with index contraction, and $\Lie_E \mathrm{Id}_{\wedge ^k\T M} = 0$, for all $E \in \mathfrak{X}(M)$, we obtain 
\[
0 = \left(\Lie_{E} \sksym{k}\widehat{B}^{-1}\right) \circ\sksym{k}\widehat B + \sksym{k}\widehat{B}^{-1} \circ \left({\Lie_E \sksym{k} \widehat B}\right)\,,
\]
from which the desired equality follows. The identity for the transpose maps follows from the exact same argument.
\end{proof}
With this, we can prove the following formula:
\begin{proposition}
\label{prop:dual_bracket_computation}
Let $\alpha\in \Omega^p(M)$ and $\beta \in \Omega^1(M)$. Then
\begin{equation}
  \llbracket\alpha, \beta\rrbracket =  \left({\Lie_{X_\beta} \sksym{p} \transp{\widehat B}}\right)\left(\sksym{p}\transp{\widehat B}^{-1} (\alpha)\right) - \Lie_{X_\beta} \alpha\,.  
\end{equation}
where $X_\beta = \transp{\widehat B}^{-1}(\beta)$.

\end{proposition}
\begin{proof} By definition, $\llbracket\alpha, \beta\rrbracket = \sksym{p}\transp{\widehat{B}} \left(\left[ \sksym{p}\transp{\widehat B}^{-1 }(\alpha), X_\beta \right]\right)$. We may write this as 
\[
\llbracket\alpha, \beta\rrbracket = - \sksym{p}\transp{\widehat{B}} \left( \Lie_{X_{\beta}} \left(\sksym{p}\transp{\widehat B}^{-1 }(\alpha)\right)\right)\,.
\]
Since the Lie derivative commutes with index contraction, we have that 
$$ \Lie_{X_{\beta}} \left(\sksym{p}\transp{\widehat B}^{-1 }(\alpha)\right) = \left(\Lie_{X_{\beta}} \sksym{p}\transp{\widehat B}^{-1 }\right)(\alpha) + \sksym{p}\transp{\widehat B}^{-1}(\Lie_{X_\beta} \alpha)\,, $$
and so, after applying Lemma \ref{lemma:Lie_derivative_of_inverse}, we get 
\[
\Lie_{X_{\beta}} \left(\sksym{p}\transp{\widehat B}^{-1}(\alpha)\right) = - \left(\sksym{p}\transp{\widehat{B}}^{-1} \circ \left({\Lie_{X_\beta} \sksym{p} \transp{\widehat B}}\right)\circ\sksym{p}\transp{\widehat B}^{-1} \right)(\alpha) + \sksym{p}\transp{\widehat B}^{-1}(\Lie_{X_\beta} \alpha)\,.
\]
Applying $-\sksym{p}\transp{\widehat{B}}$ to the previous equation we obtain 
\begin{equation}
    \label{eq:bracket_1}
  \llbracket\alpha, \beta\rrbracket =  \left({\Lie_{X_\beta} \sksym{p} \transp{\widehat B}}\right)\left(\sksym{p}\transp{\widehat B}^{-1} (\alpha)\right) - \Lie_{X_\beta} \alpha\,.  
\end{equation}
\end{proof}

By the Leibniz rule, the map
$$\left({\Lie_{X_\beta} \sksym{p} \transp{\widehat B}}\right)\circ\left(\sksym{p}\transp{\widehat B}^{-1}\right)\colon \sksym{p} \T^* M \longto \sksym{p}\T^* M\,,$$
is exactly the derivation of degree 0 on the exterior algebra, induced by the endomorphism $\widehat K = \widehat{B}^{-1} \circ \Lie_{X_\beta} \widehat{B} \in \mathrm{End}(\T M)$, which we denote by $D_{\widehat K}$ (see~\cite[p.~110]{gre1978}).
That is, for any $\alpha_1,\dots, \alpha_p \in \T ^*M$ we have
\begin{align}
    \left({\Lie_{X_\beta} \sksym{p} \transp{\widehat B}}\right)\circ\left(\sksym{p}\transp{\widehat B}^{-1}\right)(\alpha_1\wedge\dots\wedge\alpha_p) 
    &= 
    \left({\Lie_{X_\beta} \sksym{p} \transp{\widehat B}}\right)\left(\transp{\widehat B}^{-1}(\alpha_1)\wedge \dots \wedge \transp{\widehat B}^{-1}(\alpha_p)\right)\\
    &=
    \sum_{i=1}^p \alpha_1\wedge\dots\wedge \left(\Lie_{X_\beta}\transp{\widehat B}\right)\left(\transp{\widehat B}^{-1}(\alpha_i)\right)\wedge \dots\wedge \alpha_p\,.
\end{align}
Thus, we can write
$$\llbracket\alpha, \beta\rrbracket = D_{\widehat K}{\alpha} - \Lie_{X_{\beta}}\alpha\,,$$
with $\widehat K=\widehat B^{-1}\circ \Lie_{X_\beta}\widehat B$.

\begin{corollary}\label{cor:cool_formula}
    Let $\omega$ and $S$ denote, respectively, the skew-symmetric and symmetric parts of a non-degenerate 2-covariant tensor field $B$. Then, for any $\beta\in \Omega^1(M)$, we have
    $$\llbracket\omega, \beta\rrbracket = 2(\widehat \omega\circ \widehat B^{-1} \circ \Lie_{X_{\beta}} \widehat B)^{\mathrm{skew}}- \Lie_{X_\beta}\omega =- 2(\widehat S\circ \widehat B^{-1} \circ \Lie_{X_{\beta}}\widehat B)^{\mathrm{skew}}+
    \Lie_{X_\beta}\omega  \,, $$
where $\widehat T^{\mathrm{skew}}$ denotes the skew-symmetric part of the tensor $T$ associated to the linear map $\widehat T$.
\end{corollary}
\begin{proof}
For the first equality it is enough to show that if $K$ is any endomorphism and $\omega$ is a 2-form then
    $$D_{\widehat K} \omega = 2(\widehat\omega \circ \widehat K)^{\mathrm{skew}}\,.$$
Indeed, let $X, Y \in \mathfrak{X}(M)$. Since $D_{\widehat K}$ acts as a derivation of degree 0 on the exterior algebra, its action on a 2-form $\omega$ is given by
$
(D_{\widehat {K}} \omega)(X, Y) = \omega(\widehat K(X), Y) + \omega(X, \widehat K(Y))
$.
On the other hand, we have $(\widehat \omega \circ \widehat  K)(X, Y) = \omega(\widehat K(X), Y)$. Hence,
\begin{align*}
2(\widehat \omega \circ \widehat K)^{\mathrm{skew}}(X, Y) &= (\widehat \omega \circ \widehat K)(X, Y) - (\widehat \omega \circ \widehat K)(Y, X) \\
&= \omega(\widehat K(X), Y) - \omega(\widehat K(Y), X)= \omega(\widehat K(X), Y)+\omega(X, \widehat K(Y))\,.
\end{align*}
as we wanted to see.
For the other equality we use that
    $$\widehat{S} \circ \widehat  K = (\widehat{B} - \widehat{\omega}) \circ \widehat  K = \widehat{B}\circ \widehat K - \widehat{\omega} \circ \widehat K =\Lie_{X_\beta} \widehat{B} - \widehat{\omega} \circ\widehat  K\,,$$
    which, taking the skew-symmetric component, yields 
    $$ (\widehat{S} \circ \widehat K)^{\mathrm{skew}} = \Lie_{X_\beta} \widehat{\omega} - (\widehat{\omega} \circ \widehat K)^{\mathrm{skew}} \,.$$
    Substituting this relation into the first equality concludes the proof.
\end{proof}

\section{When does a 2-covariant tensor field induce a Poisson or Jacobi structure}
\label{Section:One_tensor}

\subsection{A formula for the Schouten--Nijenhuis bracket 
under direct sum decomposition}

Let $B$ be a non-degenerate 2-covariant tensor such that $\omega$ and $S$ denote, respectively, its skew-symmetric and symmetric parts.
Additionally, let us assume that $\T M = \H \oplus \V$, where 
$$
\H = \Ker S,
\qquad 
\V = \Ker \omega
$$
are vector subbundles that will be called
\emph{horizontal} and \emph{vertical}
just to facilitate the wording.
Under this hypothesis
the ranks of the skew-symmetric and symmetric parts of~$B$ are constant. 
Recall that if a 2-covariant tensor satisfies this direct sum condition, then it is necessarily non-degenerate (see Proposition~\ref{prop:inner_product_non_deg}). 

Let $\pr_\H$ and $\pr_\V$ denote the corresponding projections onto the horizontal and vertical subbundles; their natural extension to vector fields will be written using the same notation. 
Given a vector field $X\in \X(M)$ we will denote $X_\H = \pr_\H (X)$ and $X_\V = \pr_\V (X)$. 
In this section we will use the results explained in Section~\ref{sec:decomposition}.

\begin{lemma}
    Let $B$ be a non-degenerate 2-covariant tensor and denote by $\omega$ and $S$, respectively, its skew-symmetric and symmetric parts. If $\T M = \H \oplus \V$, then for every $X, Y, Z \in \X (M)$
    $$ \big\langle (\widehat \omega \circ \widehat {B}^{-1} \circ \Lie_{Z}\widehat  {B})(X), Y\big\rangle = (\Lie_{Z} B)(X, Y_\mathcal{H})\,.$$
    Similarly,
    $$\big\langle (\widehat S \circ\widehat  {B}^{-1} \circ \Lie_{Z} \widehat{B})(X), Y\big\rangle = (\Lie_{Z} B)(X, Y_\mathcal{V})\,.$$
\end{lemma}
\begin{proof}
    For any $X,Y \in \X (M)$ we have
    $$\langle \widehat\omega \circ \widehat{B}^{-1} \circ \Lie_{Z} \widehat{B}(X), Y\rangle=
    -\langle \Lie_{Z} \widehat{B}(X), \transp{\widehat{B}^{-1}}\circ\widehat\omega (Y)\rangle\,,
    $$
    and using that $\transp{\widehat{B}^{-1}}\circ\widehat\omega=-\pr_\H$ we obtain the result. The proof of the other identity is analogous.
\end{proof}

Hence, taking the skew-symmetric part and using Corollary~\ref{cor:cool_formula}, we have:
\begin{corollary}
       Let $B$ be a non-degenerate 2-covariant tensor and denote by $\omega$ and $S$, respectively, its skew-symmetric and symmetric parts. If $\T M = \H \oplus \V$, then for every $X, Y \in \X (M)$ and $\beta \in \Omega^1(M)$
\begin{equation}\label{eq:lemma_horizontal}
    \llbracket\omega,\beta \rrbracket(X,Y)= (\Lie_{X_\beta } B)(X, Y_\mathcal{H})-(\Lie_{X_\beta } B)(Y, X_\mathcal{H}) -(\Lie_{X_\beta}\omega)(X,Y)\,,
\end{equation}
where $X_\beta = \transp{\widehat B}^{-1}(\beta)$.
Similarly,
$$\llbracket\omega,\beta \rrbracket(X,Y)= -(\Lie_{X_\beta} B)(X, Y_\mathcal{V}) + (\Lie_{X_\beta} B)(Y, X_\mathcal{V}) + (\Lie_{X_\beta}\omega)(X,Y)\,.$$

\end{corollary}

With this last result we can prove the following theorem:

\begin{theorem}
\label{thm:good-looking_formula}
Let $B$ be a non-degenerate 2-covariant tensor and let $\omega$ and $S$ denote, respectively, its skew-symmetric and symmetric parts.
    Assume that $\T M = \H \oplus \V$, with $\H = \Ker S$ and $\V = \Ker \omega$ vector subbundles. 
    Then, for any $X,Y,Z \in \X (M)$ we have
    \begin{align}
    \label{eq:good_formula}
        \frac{1}{2}\llbracket\omega,\omega\rrbracket(X,Y,Z) =&\ 
        \d \omega ( X_\H, Y_\H, Z_\H)\\ & + S([X_\H,Y_\H],Z_\V) + S([Y_\H,Z_\H],X_\V) + S([Z_\H,X_\H],Y_\V)\,.
    \end{align}
\end{theorem}

\begin{proof}
    Let $f \in \Cinfty(M)$ and $X_f={\widehat B^{-1}}(\d f)$
    the associated vector. 
    If we denote by $\beta\coloneqq\inn{X_f}\omega $ then the associated vector field is $X_\beta = \transp{\widehat B }^{-1}(\widehat\omega(X_f))$ and, using that $\transp{\widehat B}^{-1}\circ\widehat \omega = -\pr_\H$, we obtain
    $$X_{\beta} = -\pr_\H (X_f)\,.$$
    In particular, $X_\beta$ is a horizontal vector field.
    Also, by Equation~\eqref{eq:bracket_omega_Xf} we have
    $$\iota_{X_f}\llbracket \omega,\omega\rrbracket = -2\llbracket\omega,\iota_{X_f}\omega\rrbracket =-2\llbracket \omega, \beta \rrbracket\,.$$
    
    Now, using this last expression and  Equation~\eqref{eq:lemma_horizontal}, let us evaluate the expression for the different possible combinations of horizontal and vertical vector fields:
    \begin{enumerate}[\rm (i)]
        \item \underline{$X ,Y $ are horizontal}: Then
        $$\llbracket\omega,\beta\rrbracket(X,Y)=(\Lie_{X_\beta} B)(X, Y)-(\Lie_{X_\beta} B)(Y, X)-(\Lie_{X_\beta}\omega)(X,Y)\,.$$
    and using that $S$ is symmetric and $B=\omega +S$, 
    the equation reduces to
    $\llbracket\omega,\beta\rrbracket(X,Y)=(\Lie_{X_\beta}\omega)(X,Y)$ $=(\inn{X_\beta}\d\omega)(X,Y)+ (\d(\inn{X_\beta}\omega))(X,Y)$.
    Now, note that from $X_\beta= -\pr_\H (X_f)$ we obtain
    $\inn{X_f}\omega= -\inn{X_\beta}\omega$, and hence also
    \begin{equation}\label{eq:th_eq_1}
        \d(\inn{X_\beta}\omega)=-\d(\iota_{X_f}\omega)=-\d(\d f- \widehat S(X_f))=\d (\widehat S(X_f))\,.
    \end{equation}
    Now,
    \begin{equation}\label{eq:th_eq_2}
        \d(\inn{X_f} S)(X, Y) =  X(S(X_f, Y)) - Y(S(X_f, X)) - S\left(X_f, [X, Y] \right)=-S(X_f,[X,Y])\,,
    \end{equation}
    so, finally, by Equation~\eqref{eq:bracket_omega_Xf}, we obtain
    $$\frac{1}{2}\left(\inn{X_f}\llbracket\omega,\omega\rrbracket\right)(X, Y)=
    -\llbracket\omega,\inn{X_f}\omega\rrbracket(X, Y) = \d\omega(\pr_\H(X_f), X, Y) + S(X_f,[ X, Y])\,.$$
    
    \item \underline{$X,Y$ are vertical}: Then
        $$\llbracket\omega,\beta\rrbracket(X,Y) = -\Lie_{X_\beta}\omega (X,Y)=-( X_\beta\left(\omega(X, Y)\right) - \omega([X_\beta, X], Y) - \omega(X, [X_\beta, Y]))=0\,,$$
        so that
         $\frac{1}{2}(\inn{X_f}\llbracket\omega,\omega\rrbracket)(X, Y) = 0$.
         \item \underline{$X$ is horizontal and $Y$ is vertical}: Then
         $$\llbracket\omega,\beta\rrbracket(X,Y)=-(\Lie_{X_{\beta}} B)(Y,X)- (\Lie_{X_{\beta}}\omega) (X,Y)=-\Lie_{X_\beta}S(X,Y)\,,$$
         and using that 
          $$(\Lie_{X_\beta} S)(X, Y) = X_\beta\left(S(X,Y)\right) - S([X_\beta, X], Y) - S(X, [X_\beta, Y])=- S([X_\beta, X],Y)\,,$$
          we have
          $\llbracket\omega,\beta\rrbracket(X,Y)= S([X_\beta, X],Y)$.
          Hence,
           $$\frac{1}{2}(\inn{X_f}\llbracket\omega,\omega\rrbracket)(X, Y) = S([\pr_\H(X_f), X],Y)\,.$$
           If $X$ is vertical and $Y$ is horizontal then, analogously, we get
           $$\frac{1}{2}(\inn{X_f}\llbracket\omega,\omega\rrbracket)(X, Y) = -S([\pr_\H(X_f), Y],X)\,.$$
            \end{enumerate}
            Finally, since both sides of the identity are $\Cinfty(M)$-linear in $X_f$, and local vector fields of the form $X_f = \widehat{B}^{-1}(\d f)$ generate $\X(M)$ as a $\Cinfty(M)$-module, $X_f$ can be replaced by an arbitrary vector field $Z \in \X(M)$.
            Hence, for any $X,Y,Z \in\X(M)$ we have
            \begin{align}
                \frac{1}{2}\llbracket\omega,\omega\rrbracket(Z,X,Y) 
                &=
                \frac{1}{2}\Big(\llbracket\omega,\omega\rrbracket(Z,X_\H,Y_\H)+\llbracket\omega,\omega\rrbracket(Z,X_\V,Y_\V)\\
                &\quad\ +\llbracket\omega,\omega\rrbracket(Z,X_\H,Y_\V)+\llbracket\omega,\omega\rrbracket(Z,X_\V,Y_\H)\Big)\\
                &=
                \d \omega (Z_\H, X_\H, Y_\H) + S([X_\H,Y_\H],Z_\V)+S([Z_\H,X_\H],Y_\V)+S([Z_\H,Y_\H],X_\V)\,,
            \end{align}
            which concludes the proof.
\end{proof}

\subsection{An interpretation in terms of the curvature
of the horizontal distribution}

Now we want to give an interpretation of this formula in terms of a certain curvature.

First, viewing the projection as an endomorphism $\pr_\H \colon \T M \to \T M$ with image~$\H$, the induced map on forms gives
$$\left((\sksym{3}\transp{\pr_\H})(\d\omega)\right)(X,Y,Z) = \d\omega(X_\H, Y_\H, Z_\H)\,.$$

Second, recall that, 
if $\mathcal{H} \subset \T M$ is a vector subbundle,
a general notion of curvature 
(see for instance \cite[p.\ 73]{Mon_2002})
is provided by the \dfn{curvature operator}
$$
\curv_\mathcal{H} \colon \mathcal{H} \otimes \mathcal{H} \rightarrow \T M/ \mathcal{H} 
\,,
\qquad 
\curv_\H (X_\H, Y_\H) = 
[X_\H,Y_\H] \mod \H
\,,
$$
which vanishes if, and only if, $\mathcal{H}$ is integrable.
This map is well defined
since it is $\Cinfty(M)$-bilinear when applied to two ``horizontal'' sections: 
$$
[f X_\H, Y_\H] = f[X_\H, Y_\H] - (\Lie_{Y_\H} f)X_\H
\,.
$$
When we have a direct sum decomposition 
$\T M = \mathcal{H} \oplus \mathcal{V}$ 
we have a canonical isomorphism 
$\T M/ \mathcal{H} \cong \mathcal{V}$,
therefore the curvature operator of~$\H$
(and also that of~$\V$, if needed)
can be interpreted as a skew-symmetric map
$$
\curv_\mathcal{H} \colon \mathcal{H} \otimes \mathcal{H} \rightarrow \mathcal{V}
\,,\qquad
\curv_\H (X_\H, Y_\H) = 
\pr_\V[X_\H,Y_\H]
\,.
$$
This can be interpreted also as 
$$
\curv_\mathcal{H} \colon \T M \otimes \T M \longrightarrow \T M
\,,\qquad
\curv_\H (X, Y) = 
\pr_\V[X_\H,Y_\H]
\,,
$$
that is, a vector valued 2-form
$\curv_\H \in \Omega^2(M; \T M)$,
see for instance \cite[p.~73]{KMS_1993}.


Now we consider the symmetric part~$S$ of~$B$
as a vector valued 1-form 
$\widehat S \in \Omega^1(M; \T^*M)$.
The canonical pairing between 
$\T^*M$ and $\T M$ 
induces an exterior product of differential forms
with values in these vector bundles
\cite[p.~82]{Bou_1983}.
When applied to $\widehat S$ and $\curv_{\H}$
we obtain the (scalar valued) 3-form defined by
\begin{equation*}   
\begin{split}
(\widehat S \barwedge \curv_{\H}) (X,Y,Z) 
&=
\langle S(X) , \curv_{\H}(Y,Z)) \rangle \ +\ \text{permutations}
\\
&=
S(X,\curv_{\H}(Y,Z)) \ +\ \text{permutations}
\\
&=
S(X,\pr_\V [Y_\H,Z_\H]) \ +\ \text{permutations}
\\
&=
2 S(X_\V,[Y_\H,Z_\H]) \ +\ \text{cyclic permutations}
\,,
\end{split}
\end{equation*}
where in the last equality we have used that $\Ker S = \H$
and that $S$ is symmetric.
In this expression we recognize the last terms of
Eqn.~\eqref{eq:good_formula}.
%
%
Therefore, we can rewrite Theorem~\ref{thm:good-looking_formula} as follows:
\begin{theorem}
\label{th:big_theorem} 
Let $B$ be a non-degenerate 2-covariant tensor such that $\T M = \H \oplus \V$, where $\H = \Ker S$ and $\V = \Ker \omega$,
and $\omega$ and $S$ denote, respectively, the skew-symmetric and symmetric parts of $B$. Then,
\begin{equation}
\label{eq:big_formula}
\frac{1}{2}\llbracket\omega,\omega\rrbracket 
= 
(\sksym{3}\transp{\pr_\H})(\d \omega) 
+ 
\frac12 \,\widehat S \barwedge \curv_{\H}
\end{equation}
\end{theorem}

\begin{remark}
    Note that the curvature of the vertical distribution does not appear in this expression. Moreover, it is worth pointing out that this formula resembles the Cartan structure equation for a connection on a principal bundle (see, for instance,~\cite[Ch.~II, \S5]{kobayashi_nomizu_1963}). This connection to principal bundle theory and fat bundles will be made explicit in Section~\ref{section:Examples}.
\end{remark}

\begin{remark}
\label{rk:formula_S_coordinates}
As $S$ is a constant-rank symmetric tensor, it can be diagonalized locally. Thus, we can write
$$ S = \sum_{k=1}^r \lambda_k \tau_k \otimes \tau_k \,, $$
for some non-vanishing local $1$-forms $\tau_k$ and functions $\lambda_k \in \Cinfty(M)$, and $r=\rank \V$. Since the horizontal distribution is defined as $\H = \Ker S$, the $1$-forms $\tau_k$ must vanish identically on any horizontal vector field. Since $S|_{\mathcal{V}}$ is non-degenerate, the set $\{\tau _k\}$ defines a basis on $\mathcal{V}^\ast$. Let $\{R_k\}$ be its induced basis on $\mathcal{V}$. Then, the curvature of $\mathcal{H} = \Ker S$ may be written as 
\[
\curv_\mathcal{H} (H_1, H_2)= \sum_{k} \tau_k([H_1, H_2]) R_k = -\sum_k (\d \tau_k)(H_1, H_2) R_k\,,
\]
for arbitrary $H_1$ and $H_2$ horizontal vector fields. 
Then, using that
$$
(\left(\sksym{2}\transp{\pr_\H}\right)(\d\tau_k) \wedge \tau_k)(X, Y, Z) =  \d\tau_k(X_\H, Y_\H) \tau_k(Z_\V)+ \text{cyclic permutations}$$
we have 
\[
\frac12 \,\widehat S \barwedge \curv_{\H} = -\sum_k \lambda_k\left((\sksym{2}\transp{\pr_\H})(\d\tau_k)\right) \wedge \tau_k\,,
\]
and so, by Equation~\eqref{eq:big_formula},
\begin{align}\label{eq:coord_formula}
   \frac{1}{2}\llbracket\omega,\omega\rrbracket= (\sksym{3}\transp{\pr_\H})(\d \omega) - \sum_k \lambda_k  \left((\sksym{2}\transp{\pr_\H})(\d\tau_k)\right) \wedge \tau_k \,.
\end{align}
\end{remark}

\subsection{Poisson and Jacobi bivector fields}

In this section, we use Theorem~\ref{th:big_theorem} to characterise the conditions under which a non-degenerate $2$-covariant tensor field $B$ defines a Poisson or Jacobi structure. Throughout, we maintain the assumption that the tangent bundle splits as $\T M = \H \oplus \V$, with $\H \coloneqq \Ker S$ and $\V \coloneqq \Ker \omega$ vector subbundles, defined as the kernels of the symmetric and skew-symmetric parts of a non-degenerate 2-covariant tensor field $B$.

As we commented before, the condition for the bivector
$\Lambda = - \big( \sksym{2} {\widehat B}^{-1}\big)(\omega)$ to be Poisson can be written with the Schouten--Nijenhuis bracket as
$[\Lambda,\Lambda]=0$, or, equivalently, as 
$$\llbracket\omega, \omega \rrbracket=0\,,$$
with $\llbracket\cdot,\cdot\rrbracket$ the conjugate Schouten--Nijenhuis bracket associated with $B$. Thus, as a direct consequence of Theorem~\ref{th:big_theorem} we have:
\begin{theorem}\label{th:Poisson_general}
    Let $B \in \Gamma(\T^\ast M \otimes \T^\ast M)$ be a non-degenerate tensor such that $\T M = \H \oplus \V$, with $\H \coloneqq \Ker S$ and $\V \coloneqq \Ker \omega$ are vector subbundles, defined, respectively, by the symmetric and skew-symmetric parts of $B$. Then, the associated bivector field $\Lambda = - \big( \sksym{2} {\widehat B}^{-1}\big)(\omega)$ defines a Poisson structure if, and only if, the following conditions hold
    \begin{enumerate}[\rm (i)]
    
        \item The exterior derivative of $\omega$ vanishes on the horizontal distribution, so $\left(\sksym{3}\transp{\pr_\H}\right)(\d \omega)=0$.
         
        \item The horizontal distribution $\mathcal{H}$ is integrable. 
         
     \end{enumerate}
\end{theorem}
Note that we have used that the two 3-forms on the right-hand side of Equation~\eqref{eq:big_formula} have different types: specifically, $(\sksym{3}\transp{\pr_\H})(\d \omega)$ is of type $(3,0)$, while $\frac12 \,(\widehat S \barwedge \curv_{\H})$ is of type $(2,1)$, with respect to the canonical decomposition
\[
\Omega^3(M) = \bigoplus_{j=0}^3 \Gamma\left( \sksym{j} \H^\ast \otimes \sksym{3-j} \V^\ast \right),
\]
induced by the splitting $\T M = \H \oplus \V$. Therefore, they cannot cancel each other. We will also use this in the sequel.

Now, let us study the more general case of Jacobi structures. For this, we first determine when
$$\frac{1}{2}\llbracket\omega,\omega\rrbracket = \alpha \wedge \omega\,,$$
for some 1-form $\alpha$.
By Equation~\eqref{eq:good_formula}, this holds if, and only if, for any $X,Y,Z\in \X(M)$ we have
$$(\d\omega)(X_\mathcal{H}, Y_\mathcal{H}, Z_\mathcal{H}) =  (\alpha \wedge \omega)(X_\mathcal{H}, Y_\mathcal{H}, Z_\mathcal{H})\,,$$
and
$$S([X_\mathcal{H}, Y_\mathcal{H}], Z_\mathcal{V}) =  \alpha(Z_\mathcal{V})\omega(X_\mathcal{H}, Y_\mathcal{H})\,.$$
Note that this last condition implies that the image of the curvature of the horizontal distribution has dimension at most one. Indeed, if we let $X_\alpha = \transp{\widehat B}^{-1}(\alpha)$ then for any $Z_\V \in \V$ we have
$$\alpha(Z_\V)= \langle \transp{\widehat B}(X_\alpha), Z_\V \rangle = \langle \widehat B(Z_\V), X_\alpha \rangle = B(Z_\V, X_\alpha)= S((X_\alpha)_\V,Z_\V)\,,$$
where we have used that $B|_\V = S|_\V$. Since the restriction of $S$ to $\V$ is non-degenerate, we have
$$\curv_\H(X_\mathcal{H}, Y_\mathcal{H})=  \omega(X_\mathcal{H}, Y_\mathcal{H})(X_\alpha)_\V \,.$$
In particular, it is enough to look for the horizontal and vertical components of $\alpha$ which, if we denote by $\overline{\alpha} \in \Gamma(\H^\ast)$ and $\transp{\widehat{B}}(R)\in \Gamma(\V^*)$, for some $R \in \Gamma(\V)$, need to satisfy
\begin{equation}
    \left(\sksym{3}\transp{\pr_\H}\right)(\d \omega) = \overline{\alpha} \wedge \left(\sksym{2}\transp{\pr_\H}\right)(\omega) \qquad \text{and} \qquad \curv_\H(X_\H, Y_\H) = \omega(X_\H, Y_\H) R\,.
\end{equation}

Secondly, we need to check under which conditions the equality
$$\llbracket\omega,\alpha\rrbracket=0$$
holds. Recall that, for any $X,Y\in \X(M)$, we have
$$\llbracket\omega,\alpha\rrbracket(X,Y)=(\Lie_{X_\alpha} B)(X, Y_\mathcal{H})-(\Lie_{X_\alpha} B)(Y, X_\mathcal{H})-(\Lie_{X_\alpha}\omega)(X,Y)\,,$$
with $X_\alpha=\transp{\widehat B}^{-1}(\alpha)$. We may reduce to studying the following three cases:
\begin{enumerate}[\rm (i)]
    \item If $X$ and $Y$ are horizontal, then
    $$\llbracket\omega,\alpha\rrbracket(X,Y)=(\Lie_{X_\alpha}\omega)(X,Y)\,,$$
    \item If $X$ and $Y$ are vertical, then $$\llbracket\omega,\alpha\rrbracket(X,Y)=-(\Lie_{X_\alpha}\omega)(X,Y) = -X_\alpha(\omega(X,Y)) + \omega([X_\alpha, X], Y) + \omega(X, [X_\alpha, Y])=0\,.$$
    \item If $X$ is horizontal and $Y$ is vertical, then
    \begin{align}
      \llbracket\omega,\alpha\rrbracket(X,Y) &=
      -(\Lie_{X_\alpha} S )(X,Y)= S([X_\alpha, X], Y)\,.  
    \end{align} 
\end{enumerate}
After these considerations, we obtain the following:
\begin{theorem}
\label{thm:Jacobi_general}
     Let $B \in \Gamma(\T^\ast M \otimes \T^\ast M)$ be a non-degenerate tensor and suppose that $\T M = \H \oplus \V$, where $\H \coloneqq \Ker S$ and $\V \coloneqq \Ker \omega$ are vector subbundles, defined, respectively, by the symmetric and skew-symmetric parts of $B$. Then, the bivector field $\Lambda = - \big( \sksym{2} {\widehat B}^{-1}\big)(\omega)$ can be completed to a Jacobi structure if, and only if, there exists a $1$-form of type $(1,0)$, $\overline{\alpha} \in \Gamma(\H^\ast)$, and a vector field $R \in \Gamma(\V)$ such that:
     \begin{enumerate}[\rm (i)]
         \item
         $\left(\sksym{3}\transp{\pr_\H}\right)( \d\omega)=\overline{\alpha}\wedge \left(\sksym{2}\transp{\pr_\H}\right)(\omega)$,
         \item $\curv_\H (X_\H,Y_\H) =\omega (X_\H,Y_\H) R$,
         \item 
         If we write $\alpha= \overline{\alpha}+\transp{\widehat B}(R)$ and $X_\alpha= \transp{\widehat{B}^{-1}}(\alpha)$, then we have $\left(\sksym{2}\transp{\pr_\H}\right)(\Lie_{X_{\alpha}} \omega)  = 0$ and $[X_{\alpha}, \H] \subseteq \H$.
     \end{enumerate}
     The Jacobi structure is then given by $\Lambda = -\big( \sksym{2} {\widehat B}^{-1}\big)(\omega)$ and $E=X_\alpha=\transp{ \widehat B^{-1}}(\overline{\alpha})+R$.
\end{theorem}

\begin{example}
As a simple but enlightening example, we study the following problem. Let $(M, \omega)$ be a symplectic manifold and $W$ be a symplectic distribution, which yields the direct sum decomposition $\T M = W \oplus W^{\perp, \omega}$. Let $P \colon \T M \rightarrow W$ denote the natural projection onto $W$, and let $\Lambda$ denote the usual Poisson bivector field associated with $\omega$, i.e.\
$\Lambda (\alpha, \beta) = - \omega(\widehat{\omega}^{-1}(\alpha), \widehat{\omega}^{-1}(\beta))$.
This, obviously, coincides with the bivector field induced by $B=\omega$.
The question we aim to answer is: under which conditions does the projected bivector field $P(\Lambda) \in \Gamma(\sksym{2} W)$ define a Poisson structure?

It is a well-known fact \cite{bhaskara1988poisson} that this occurs if, and only if, the distribution $W$ is involutive. We may recover this classical result through our framework as follows.
Let $g$ be a Euclidean inner product on $W^{\perp, \omega}$. Using the decomposition of $\T M$, we can define a new non-degenerate 2-covariant tensor $B'$ by extending both $\omega|_W$ and $g$ to the entire tangent bundle:$$B' \coloneqq \omega|_{W} + g\,.$$
This construction satisfies the hypotheses of Theorem \ref{th:big_theorem}, and it is straightforward to verify that $P(\Lambda)$ coincides with the bivector field $\Lambda= - \Big( \sksym{2} \widehat {B'}^{-1}\Big) (\restr{\omega}{W})$.
Applying Theorem \ref{th:big_theorem}, we find that $P(\Lambda)$ is a Poisson structure if, and only if, $W$ is integrable and $\left(\sksym{3}\transp{\pr_W} \right)(\d (\omega|_W)) = 0$. To conclude the proof, we just need to observe that because the original form $\omega$ is closed, the integrability of $W$ immediately implies that $\left(\sksym{3}\transp{\pr_W} \right)(\d (\omega|_W))= 0$.
    
\end{example}

\subsection{Some remarks about the general case}

In the general case, $B \in \Gamma\left(\T^\ast M \otimes \T^\ast M\right)$ does not satisfy the hypothesis $\T M = \Ker S \oplus \Ker \omega$
of Theorem~\ref{thm:good-looking_formula}.
Nevertheless, under some regularity conditions,
we can change part of the structure so that we can apply the theorem.
First, let us point out the following:

\begin{proposition} Any bivector field $\Lambda \in \mathfrak{X}^2(M)$ can be written as $\Lambda=- \big( \sksym{2} {\widehat B}^{-1}\big)(\omega)$, for some non-degenerate tensor $B \in \Gamma(\T^\ast M \otimes \T^\ast M)$ with skew-symmetric part $\omega$. 
\end{proposition}

\begin{proof} Let $g$ be a Riemannian metric on $M$ and let $\Sigma \coloneqq g^{-1}$ be the inverse symmetric contravariant tensor. Then, $\Lambda + \Sigma$ is a non-degenerate 2-contravariant tensor by Corollary \ref{cor:inner_product_non_deg}, and it is enough to take $B \coloneqq (\Lambda + \Sigma)^{-1}$.
\end{proof}

\begin{example}
It follows from the argument above that a particular subclass of metriplectic systems induces 2-covariant tensors. A metriplectic structure on a manifold $M$ consists of a Poisson bivector $\Lambda \in \mathfrak{X}^2(M)$ and a symmetric positive semidefinite contravariant tensor $\Sigma \in \Gamma(\T M \otimes \T M)$ \cite{mor1986physicad_nonlinearphenomena}. Then, if the kernels of the linear maps induced by the tensors satisfy $\Ker \Lambda \cap \Ker \Sigma = 0$, the tensor $\Lambda + \Sigma$ is invertible (see Proposition~\ref{prop:inner_product_non_deg}) and we may build $B \coloneqq (\Lambda + \Sigma)^{-1} \in \Gamma(\T^\ast M \otimes \T^\ast M)$.
\end{example}

Also, as an immediate application of Proposition \ref{prop:existence_of_equivalent_B} we obtain the following:

\begin{proposition} Let $B \in \Gamma(\T^\ast M \otimes \T^\ast M)$ be a non-degenerate 2-covariant tensor such that its skew-symmetric part $\omega$ has constant rank. Then, there exists a non-degenerate 2-covariant tensor $B' \in \Gamma(\T^\ast M \otimes \T^\ast M)$ satisfying $\T M = \Ker \omega' \oplus \Ker S'$ and that
\[
\left( \sksym{2}{\widehat B}^{-1}\right)(\omega) = \left(\sksym{2} {\widehat {B'}}^{-1} \right)(\omega')\,.
\]
\end{proposition}

So, it follows from the last two propositions, that any bivector field $\Lambda$ with constant rank can be obtained as 
$\Lambda= -\big( \sksym{2} {\widehat B}^{-1}\big)(\omega)$
for some non-degenerate 2-covariant tensor field $B$ (with $\omega$ and $S$, respectively, its skew-symmetric and symmetric parts) satisfying the direct sum condition $\T M = \Ker \omega \oplus \Ker S$.

\section{Notable examples}
\label{section:Examples}

In this section, we apply the results obtained in the previous sections, mainly Theorem~\ref{th:big_theorem}, to some examples which are relevant in classical mechanics and mathematical physics.

\subsection{Doublets}
\label{sec:doublets}

In this first section we will study different examples where the 2-covariant tensor field $B$ is formed by a pair of a 1-form $\tau$ and a 2-form $\omega$, so that
$$B=\tau\otimes \tau+\omega \,.$$
These structures have been studied in detail recently in~\cite{GMR_2026}. Let us review briefly the most important definitions and results:
\begin{definition}
Let $\alpha\in \Omega^p(M)$ be a $p$-form on $M$ and $\beta \in \Omega^q(M)$. 
The \dfn{characteristic distribution} of the pair $(\alpha, \beta)$ at a point $x\in M$ is 
$$
\K_{(\alpha,\beta)_x} = \Ker \alpha_x \cap \Ker \beta_x \,.
$$
The \dfn{class} of $(\alpha, \beta)$ at~$x$ is the codimension of its characteristic distribution at~$x$.
$$
\cl(\alpha,\beta)_x = 
\codim \K_{(\alpha,\beta)_x} 
\,.
$$

\end{definition}

\begin{definition}\label{def:doublets}
    We say that a pair $(\tau,\omega)$ of a 1-form and a 2-form is a \dfn{doublet} if the class of $(\tau,\omega)$ is constant.
\end{definition}

The parity of the class of a pair $(\tau,\omega)$ is characterised by the existence of the following vector fields:
\begin{definition}
Let $(\tau,\omega)$ be a pair of a 1-form and a 2-form.
\begin{enumerate}[\rm (i)]
\item A \dfn{Reeb vector field} for the pair 
is a vector field $R$ satisfying
\begin{equation}
\label{reeb_definition}
\inn{R}\tau= 1\,, \qquad \inn{R}\omega=0\,. 
\end{equation}

\item A \dfn{Liouville vector field} for the pair 
is a vector field $\Delta$ satisfying
\begin{equation}
\label{Liouville_vf}
\inn{\Delta}\omega = \tau \,.
\end{equation}
\end{enumerate}
\end{definition}
 
\begin{theorem}
\label{thm:class_parity}
Let $(\tau,\omega)$ be a doublet. Then
$$\text{The class of the doublet is odd} \Longleftrightarrow
\Ker \omega \not\subset \Ker\tau
\Longleftrightarrow
\text{The doublet has a Reeb vector field}\,.
$$
\end{theorem}
\begin{theorem}
Let $(\tau,\omega)$ be a doublet. Then
$$
\text{The class of the doublet is even}
\Longleftrightarrow
\Ker \omega \subset \Ker \tau
\Longleftrightarrow
\text{The doublet has a Liouville vector field}\,.
$$
\end{theorem}
Additionally, if we let $B=\omega + \tau \otimes \tau$, then both Reeb and Liouville vector fields are characterised as the elements of the preimage~$\widehat B^{-1}(\tau)$. That is, if the class is even then
$$\widehat B(X) = \tau \Longleftrightarrow X\text{ is a Liouville vector field}\,,$$
and if the class is odd then
$$\widehat B(X) = \tau \Longleftrightarrow X\text{ is a Reeb vector field}\,.$$

For our purposes, we need to study when the associated tensor field $B=\omega +\tau\otimes \tau$ is non-degenerate and induces a direct sum decomposition.

It turns out that $B$ is non-degenerate if, and only if, the pair $(\tau,\omega)$ has maximal class, meaning its class equals the dimension of the manifold. This is due to the equality
 $$\Ker B = \Ker \tau \cap\Ker \omega\,.$$
Note that, if the class is maximal, then, necessarily, the rank of $\omega$ is also maximal. In particular, $\cl(\tau,\omega)=\dim M = 2n$ if, and only if, $\omega$ is non-degenerate.
Also, one has
 \begin{proposition}
\label{prop:directsum}
Let $(\tau,\omega)$ be a pair of a 1-form and a 2-form. 
The following are equivalent:
\begin{enumerate}[$(i)$]
\item 
$\T^* M =(\Ker \tau)^{\circ} \oplus (\Ker \omega)^{\circ}$;
\item
$\T M = \Ker \tau \oplus \Ker \omega $;
\item 
$\cl (\tau,\omega) = \dim M$ (equivalently, $B$ is non-degenerate) and either the class is odd or $\tau=0$.
\end{enumerate}
\end{proposition}

So, as the kernel of the symmetric part of $B$ coincides with $\Ker \tau$, we have:
\begin{enumerate}[\rm (i)]
    \item If $\dim M =2n+1$, then the tensor field $B=\omega + \tau\otimes \tau$ associated with a pair $(\tau,\omega)$ is non-degenerate and induces a direct sum decomposition $\T M= \Ker \tau \oplus \Ker \omega$ if, and only if, $\cl (\tau,\omega)=2n+1$.
    \item If $\dim M =2n$, then the tensor field $B=\omega + \tau\otimes \tau$ associated with a pair $(\tau,\omega)$ is non-degenerate and induces a direct sum decomposition $\T M= \Ker \tau \oplus \Ker \omega$ if, and only if, $\omega$ is non-degenerate and $\tau=0$.
\end{enumerate}

If $\dim M = 2n$ and $\omega$ is non-degenerate, we can reduce the problem to the direct sum decomposition case for an arbitrary $1$-form $\tau$. This is because the bivector field associated with the non-degenerate tensor $B = \omega + \tau \otimes \tau$ is independent of the 1-form $\tau$.
Indeed, since $\omega$ is non-degenerate when the class is even and equal to the dimension of the manifold, we can define the standard bivector associated with $\omega$ (equivalently, the bivector field associated with $B' = \omega$) as
$$\Lambda'(\alpha, \beta) = -\omega(\widehat{\omega}^{-1}(\alpha), \widehat{\omega}^{-1}(\beta)) \,,$$
for any $\alpha,\beta \in \Omega^1(M)$.
To see why the bivector associated with $B$ coincides with $\Lambda'$, we use the following identity for the inverse 
$$\widehat{B}^{-1}(\alpha) = \widehat{\omega}^{-1}(\alpha) + (\iota_\Delta \alpha)\Delta \,,$$
where $\alpha \in \Omega^1(M)$ and $\Delta \in \mathfrak{X}(M)$ denotes the unique Liouville vector field of the doublet $(\tau, \omega)$, which has class $2n$. 
We can verify that this identity holds by directly evaluating $\widehat{B}$ on the proposed inverse. Recalling that, for any $X\in \X(M)$, $\widehat{B}(X) = \widehat{\omega}(X) + \tau(X)\tau$ and that $\iota_\Delta \tau=0$, we obtain
$$\widehat{B} \left( \widehat{\omega}^{-1}(\alpha) + (\iota_\Delta \alpha)\Delta \right) = \alpha + (\iota_\Delta \alpha)\widehat{\omega}(\Delta) + (\iota_{\widehat{\omega}^{-1}(\alpha)}\tau)\tau \,.$$
Using that, by definition, $\widehat{\omega}(\Delta) = \tau$ and that $\omega$ is skew-symmetric, we have
$$\iota_{\widehat{\omega}^{-1}(\alpha)}\tau = \iota_{\widehat{\omega}^{-1}(\alpha)}\widehat{\omega}(\Delta) = -\iota_\Delta \alpha \,.$$
Substituting this into our expression yields
$$\widehat{B} \left( \widehat{\omega}^{-1}(\alpha) + (\iota_\Delta \alpha)\Delta \right)=\alpha\,.$$
Since $\widehat{B}^{-1}$ differs from $\widehat{\omega}^{-1}$ only by the purely symmetric term $(\iota_\Delta \alpha)\Delta$, their skew-symmetric parts are identical. Consequently, the associated bivector fields coincide, and the result is independent of $\tau$.
Hence, even if the dimension of the manifold is even and the class of the doublet is $2n$, we may always reduce to the case
$$\T M = \Ker \tau \oplus \Ker \omega\,, $$
by taking $\tau=0$.

With this, and the fact that 
$$ \curv_\H(X_\H,Y_\H) = \tau([X_\H,Y_\H])R = -\d\tau(X_\H,Y_\H)R \,,$$
if the doublet has maximal odd class, 
we can apply Theorem~\ref{th:big_theorem} to obtain the conditions under which a doublet defines a Poisson or Jacobi manifold:
\begin{theorem}[Poisson and Jacobi in even class]\label{th:poi_jac_even} 
Let $(\tau, \omega)$ be a doublet such that $\cl(\tau,\omega)= \dim M=2n$. Then,
\begin{enumerate}[\rm (i)]
    \item The 2-covariant tensor $B = \omega + \tau \otimes \tau$ is non-degenerate and defines a Poisson structure if, and only if, $\d \omega = 0$.
    \item The 2-covariant tensor field $B = \omega + \tau \otimes \tau$ is non-degenerate and defines a Jacobi structure if, and only if, $\d \omega = \alpha \wedge \omega$, for some closed $1$-form $\alpha$.
\end{enumerate}
\end{theorem}


\begin{theorem}[Poisson and Jacobi in odd class]
\label{thm:odd_class_brackets}
Let $(\tau, \omega)$ be a doublet such that $\cl(\tau,\omega) =\dim M = 2n+1$. If we denote by $\mathcal{H} \coloneqq \Ker \tau$ and by $R$ the unique Reeb vector field associated to the doublet, then
\begin{enumerate}[\rm (i)]
   \item The non-degenerate $2$-covariant tensor field $B = \omega + \tau \otimes \tau$ defines a Poisson structure if, and only if, the distribution $\mathcal{H}$ is integrable (this occurs precisely when $\tau \wedge \d\tau = 0$) and $\left(\sksym{3}\transp{\pr_\H}\right)(\d \omega) = 0$ (which is equivalent to $\d \omega \wedge \tau = 0$).
   
    \item The non-degenerate 2-covariant tensor field $B = \omega + \tau \otimes \tau$ defines a Jacobi structure if, and only if, there exists a function $\lambda \in \Cinfty(M)$ and a $1$-form $\overline{\alpha} \in \Gamma(\H^\ast)$ such that the following conditions hold
    \begin{enumerate}[\rm (a)]
        \item $\left(\sksym{3}\transp{\pr_\H}\right)(\d \omega) = \overline{\alpha} \wedge \left(\sksym{2}\transp{\pr_\H}\right)(\omega)$ or, equivalently, $\d \omega\wedge \tau = \overline{\alpha} \wedge \omega\wedge \tau$
        \item $\left(\sksym{2}\transp{\pr_\H}\right)(\d \tau)  = \lambda \cdot \left(\sksym{2}\transp{\pr_\H}\right)(\omega)$, or, equivalently, $\d \tau\wedge \tau = \lambda (\omega\wedge \tau)$,
        \item 
        Let $\alpha=\bar\alpha-\lambda \transp{\widehat B}(R)$, and $X_\alpha = \transp{\widehat B}^{-1}(\alpha)$, with $R$ the Reeb vector field of $(\tau,\omega)$. Then
        \[
        \left(\transp{\pr_\H}\right)(\inn{X_{\alpha}} \d \tau )=\left(\transp{\pr_\H}\right) \d \lambda \qquad \text{and} \qquad \left(\sksym{2}\transp{\pr_\H}\right)(\Lie_{X_\alpha}\omega) = 0\,,
        \]
        equivalently, $(\iota_{X_\alpha}\d\tau-\d\lambda)\wedge \tau=0$ and $(\Lie_{X_\alpha}\omega) \wedge \tau=0$.
    \end{enumerate}
    Then the Jacobi structure is given by $\Lambda = - \big( \sksym{2} {\widehat B}^{-1}\big) (\omega) $ and $E=X_\alpha = \transp{\widehat B}^{-1}(\overline{\alpha}) -\lambda R$.
\end{enumerate}
\end{theorem}

These last two theorems are just a particular case of Theorem~\ref{th:Poisson_general} and \ref{thm:Jacobi_general}, when applied to doublets. We also used the fact that a form $\gamma$ vanishes on $\Ker \tau$ if, and only if, it is divisible by $\tau$, i.e.\ it can be written as $\gamma=\beta\wedge \tau$ (or equivalently, $\gamma \wedge \tau = 0$).

Now, let us use the results from the last sections to recover all the usual brackets from different classical geometric structures. 

\subsubsection{Symplectic structures}

A \dfn{symplectic form} on a manifold $M$ is a 2-form $\omega\in \Omega^2(M)$ such that
\begin{enumerate}[\rm (i)]
\item 
it is closed:
$\d \omega = 0$, and
\item 
it is non-degenerate.
\end{enumerate}
It follows from the non-degeneracy condition that $\omega$ defines a non-degenerate 2-covariant tensor field $B=\omega$. As $\Ker \omega = {0}$ and the symmetric part is $S=0$ (so its kernel is the whole space) we have a direct sum decomposition of $\T M= \Ker \omega \oplus \Ker S$ and we can apply Equation~\eqref{eq:coord_formula} to compute the conjugate Schouten--Nijenhuis bracket associated with $\omega$, as
$$\frac{1}{2}\llbracket \omega, \omega\rrbracket=\d \omega = 0\,,$$
so, indeed, the bivector field $
\Lambda =- \big( \sksym{2} {\widehat B}^{-1}\big) (\omega)
$ associated with $\omega$ is a Poisson bivector field.

Furthermore, it follows from Theorem~\ref{th:poi_jac_even} that a pair $(\tau,\omega)$ with $\omega$ non-degenerate defines a Poisson structure if, and only if, $\omega$ is symplectic (i.e.\ it is also closed). This is because if $\omega$ is non-degenerate then, necessarily, $\dim M= 2n$ and $\cl(\tau,\omega)=2n$ for any $\tau\in \Omega^1(M)$.

\subsubsection{Locally conformally symplectic structures}

A $2$-form $\omega$ is said to be \dfn{conformally symplectic} 
if there exists a function $f\in \Cinfty(M)$ such that 
$\mathrm{e}^f \omega$
is a symplectic form.
In other terms:
\begin{enumerate}[\rm (i)]
\item 
$\d (e^f \omega) = 0$ and
\item 
$\omega$ is non-degenerate.
\end{enumerate}
A 2-form $\omega$ is \dfn{locally conformally symplectic}
if this conformal factor can be found locally around each point.
Equivalently (see~\cite{CG_2025}), $\omega$ is locally conformally symplectic if, and only if, there exists a \emph{closed} $1$-form $\theta$ 
(the so-called \dfn{Lee form}) 
such that
$$
\d \omega = - \theta \wedge \omega\,.
$$
The 2-form is (globally) conformally symplectic if, moreover, $\theta$ is exact.

Again, the non-degeneracy condition, implies that $B= \omega$ is a non-degenerate 2-covariant tensor field such that $\T M = \Ker \omega \oplus \Ker S = \Ker S$. Applying Formula~\eqref{eq:coord_formula} yields
$$\frac{1}{2}\llbracket \omega, \omega \rrbracket = \d \omega = -\theta \wedge \omega\,,$$
and, using Equation~\eqref{eq:lemma_horizontal}, we have
$$\llbracket \omega, -\theta \rrbracket = \Lie_{\Delta}\omega=\iota_{\Delta}\d \omega + \d (\iota_\Delta \omega) =-\iota_\Delta(\theta\wedge \omega) +\d \theta=0\,,$$
where $\Delta = \transp{\widehat \omega}^{-1}(-\theta)=-\widehat \omega^{-1}(-\theta)$ is the Liouville vector field of the doublet $(\theta,\omega)$, so it satisfies $\iota_\Delta\omega= \theta$ and $\iota_\Delta \theta = 0$.
Thus, we have shown that every locally conformally symplectic manifold defines a Jacobi bracket, with $\Lambda =- \big( \sksym{2}{\widehat B}^{-1}\big) (\omega)$ and associated vector field $E=\transp{\widehat \omega}^{-1}(-\theta)=\Delta$.

Note that it follows from Theorem~\ref{th:poi_jac_even} that a non-degenerate $2$-form $\omega$ defines a Jacobi structure if, and only if, it is locally conformally symplectic.


\subsubsection{Contact structures}


A \dfn{contact form} on a $(2n+1)$-dimensional manifold $M$ is a 1-form $\eta$ such that 
\begin{enumerate}[\rm (i)]
    \item $\eta \wedge (\d\eta)^n$ is a volume form on $M$, i.e.\ it is nowhere vanishing.
\end{enumerate}
This is equivalent to the pair $(\eta, \d\eta)$ being a doublet of class $2n+1$ and to the 2-covariant tensor field
$B = \d \eta + \eta \otimes \eta\,,$
being non-degenerate. With such a structure, the distributions $\H = \Ker \eta$ and $\V = \Ker \d \eta$ satisfy $\T M = \H \oplus \V$, and so we can apply Equation~\eqref{eq:coord_formula} to obtain
$$\frac{1}{2}\llbracket\d\eta ,\d\eta\rrbracket = +\left(\sksym{3}\transp{\pr_\H}\right)(\d(\d\eta))-\left(\sksym{2}\transp{\pr_\H}\right)(\d \eta) \wedge \eta = -\d\eta \wedge \eta=-\eta \wedge \d \eta\,.$$
Also, from $\transp{\widehat B}^{-1}(-\eta)=-R$, with $R$ the Reeb vector field of the contact structure, and the fact that $\Lie_R \d \eta =\Lie_R \eta=0$ it follows from \eqref{eq:lemma_horizontal} that
$$\llbracket\d\eta,-\eta\rrbracket= 0\,.$$
So every contact structure defines a Jacobi structure, with $\Lambda =- \big( \sksym{2} {\widehat B}^{-1}\big) (\d \eta )$ and $E = \transp{\widehat B}^{-1}(-\eta )= -R$.

\subsubsection{Cosymplectic and partially cosymplectic structures}

A pair $(\tau,\omega)$ of a 1-form and a 2-form on a $(2n+1)$-dimensional manifold $M$ is called a \dfn{partially cosymplectic structure}~\cite{lb2025j.phys.a_math.theor.} if
\begin{enumerate}[\rm (i)]
    \item the 2-form is closed: $\d\omega = 0$, and
    \item $\tau\wedge \omega^n\neq 0 $ is a volume form on $M$, i.e.\ it is nowhere vanishing.
\end{enumerate}
If, additionally, $\d \tau=0$ (so both forms are closed), then  $(\tau,\omega)$ is a \dfn{cosymplectic structure}.
The volume form condition is equivalent (see~\cite{GMR_2026}) to the class of $(\tau,\omega)$ being equal to $2n+1$ or, also, to the 2-covariant tensor field $B=\omega+ \tau \otimes \tau$ being non-degenerate. By Proposition~\ref{prop:directsum}, the tensor field $B$ induces a direct sum decomposition $\T M = \H \oplus \V$, with $\H = \Ker \tau $ and $\V = \Ker \omega$. 

If, for a partially cosymplectic structure, we apply Equation~\eqref{eq:coord_formula}, and the fact that $\d \omega = 0$, then we obtain
$$\frac{1}{2}\llbracket\omega,\omega \rrbracket = \left(\sksym{3}\transp{\pr_\H}\right)(\d \omega) - \left(\sksym{2}\transp{\pr_\H}\right)(\d \tau) \wedge \tau=-\left(\sksym{2}\transp{\pr_\H}\right)(\d \tau) \wedge \tau=-\d\tau \wedge \tau\,.$$
Thus, for a cosymplectic structure, as $\d\tau =0$, the associated bivector field $\Lambda = - \big( \sksym{2} {\widehat B}^{-1}\big) (\omega)$ is Poisson. If, on the other hand, the structure is just partially cosymplectic, then we need that
$$-\d \tau \wedge \tau = \alpha \wedge \omega\,,$$
for some 1-form $\alpha$.
The remaining condition would be
$$\llbracket \omega, \alpha\rrbracket = 0\,,$$
which will be investigated thoroughly in Section \ref{subsection:order_p}.

\subsection{Cocontact structures}

A cocontact manifold~\cite{LGGMR_23} is a triple $(M, \eta, \tau)$, where $M$ is a $(2n +2)$-dimensional manifold, $\tau$ is a closed $1$-form $\tau \in \Omega^1(M)$ and $\eta$ is a $1$-form satisfying $(\d \eta)^n \wedge \eta \wedge \tau \neq 0$.
In this case, the 2-covariant tensor field
$$B = \d\eta + \eta \otimes\eta + \tau\otimes\tau\,,$$
is non-degenerate and it induces a direct sum decomposition
$\T M = \H \oplus \V$, with $\V =\Ker \d \eta$ and $\H = \Ker S= \Ker \eta\cap \Ker \tau $. If we apply Equation~\eqref{eq:coord_formula} we have
$$\frac{1}{2}\llbracket\d\eta ,\d\eta \rrbracket=\left(\sksym{3}\transp{\pr_\H}\right)(\d(\d\eta)) -\left(\sksym{2}\transp{\pr_\H}\right)(\d\tau) \wedge \tau -\left(\sksym{2}\transp{\pr_\H}\right)(\d\eta) \wedge \eta = -\eta \wedge \d\eta\,.$$
And, similarly to the contact case, we have $\transp{\widehat B}^{-1}(-\eta)=-R$, with $R$ the Reeb vector field associated with $\eta$, which satisfies $\iota_R \d\eta =\iota_R \tau = 0$ and $\iota_R\eta = 1$. So  $\Lie_R \d \eta =\Lie_R \eta=\Lie_R\tau=0$, and therefore
$$\llbracket\d \eta , -\eta \rrbracket= 0\,.$$
Thus, any cocontact structure defines a Jacobi bracket with $\Lambda = - \big( \sksym{2} {\widehat B}^{-1}\big) (\d\eta) $ and $E= -R$.

\subsection{Almost cosymplectic structures of order \texorpdfstring{$p$}{}}
\label{subsection:order_p}

An almost cosymplectic structure of order $p$ \cite{lb2025j.phys.a_math.theor.} on a $(2n + p)$-dimensional manifold $M$ is a tuple $(\omega, \tau_1, \dots, \tau_p)$, where $\omega \in \Omega^2(M)$ and $\tau_1, \dots, \tau_p \in \Omega^1(M)$ satisfy the non-degeneracy conditions:
\[
\omega^n \wedge \tau_1 \wedge \cdots \wedge \tau_p \neq 0\,, \qquad \text{and} \qquad  \omega^{n+1}=0\,.
\]
In this scenario we can consider the non-degenerate 2-covariant tensor
\[
B = \omega + \sum_{k = 1}^p \tau_k \otimes \tau_k\,.
\]
The conditions on the differential forms impose that $\Ker \omega$ has dimension $p$, that $\Ker S = \cap_k\Ker \tau_k$ has dimension $2n$, and that these two distributions are in direct sum.
Hence, we can apply the theory developed in this paper to characterise under which conditions the bivector field 
\[
\Lambda(\alpha, \beta) = - \omega(\widehat{B}^{-1}(\alpha), \widehat{B}^{-1}(\beta))
\]
defines a Poisson or Jacobi structure.
Recall (see Remark~\ref{rk:formula_S_coordinates}) that we have 
\[
\curv_\H(H_1, H_2) = -\sum_k (\d \tau_k)(H_1, H_2) R_k\,,
\]
where $R_1, \dots, R_p$ denote the unique vector fields satisfying $\inn{R_k} \omega = 0$ and $\tau_j(R_k) = \delta_{jk}$. Then, as a straightforward consequence of the previous study we obtain:

\begin{proposition} The bivector field $\Lambda = -\big( \sksym{2} {\widehat B}^{-1}\big)(\omega)$ associated with the non-degenerate 2-covariant tensor field $B = \omega + \sum_{k=1}^p \tau_k \otimes \tau_k$ of an almost cosymplectic structure $(\omega, \tau_1,\dots, \tau_p)$ of order $p$ defines a Poisson structure if, and only if, 
\begin{enumerate}[\rm (i)]
    \item $\left(\sksym{3}\transp{\pr_\H}\right)( \d \omega) = 0$.
    \item For every $k$, we have $\left(\sksym{2}\transp{\pr_\H}\right)(\d \tau_k) = 0 $.
\end{enumerate}
\end{proposition}

\begin{proposition} The bivector field $\Lambda = -\big( \sksym{2} {\widehat B}^{-1}\big)(\omega)$ associated with the non-degenerate 2-covariant tensor field $B = \omega + \sum_{k=1}^p \tau_k \otimes \tau_k$ of an almost cosymplectic structure $(\omega, \tau_1,\dots, \tau_p)$ of order $p$ can be completed to a Jacobi structure if, and only if, there exists a horizontal $1$-form $\overline{\alpha} \in \Gamma(\H^*)$ and a set of functions $\lambda_k \in \Cinfty(M)$ such that:
\begin{enumerate}[\rm (i)]
    \item $\left(\sksym{3}\transp{\pr_\H}\right)(\d \omega) = \overline{\alpha} \wedge \left(\sksym{2}\transp{\pr_\H}\right)(\omega)$;
    \item the functions $\lambda_k$ satisfy $\left(\sksym{2}\transp{\pr_\H}\right)(\d \tau_k)= \lambda_k \cdot \left(\sksym{2}\transp{\pr_\H}\right)(\omega)$;
    \item let $X_\alpha = X_{\overline{\alpha}} - \sum_{k = 1}^{p} \lambda_k R_k$ with $X_{\overline{\alpha}} = \transp{\widehat B}^{-1}(\overline{\alpha})$ and $R_k=\widehat B^{-1}(\tau_k)$, then for every $k = 1, \dots, p$:
    \[
    \left(\transp{\pr_\H}\right)(\inn{X_{\alpha}} \d \tau_k )= \left(\transp{\pr_\H}\right)(\d \lambda_k) \qquad \text{and} \qquad \left(\sksym{2}\transp{\pr_\H}\right)(\Lie_{X_\alpha}\omega) = 0\,,
    \]
    or, equivalently, $(\inn{X_\alpha}\d\tau_k - \d\lambda_k)\wedge \tau_1 \wedge \cdots \wedge \tau_p = 0$ and $(\Lie_{X_\alpha} \omega)\wedge \tau_1 \wedge \cdots \wedge \tau_p = 0$, respectively.
\end{enumerate}
Under these conditions, the Jacobi structure is given by $\Lambda = - \big( \sksym{2} {\widehat B}^{-1}\big) (\omega) $ and $E= X_\alpha$.
\end{proposition}

Note that the cocontact structures presented in the previous section are a particular case of almost cosymplectic structures of order 2, and contact structures are of order 1.

\subsection{Fat bundles}

This section establishes a precise relationship between the formula
\begin{equation}\label{eq:H_V_equation}
    \frac{1}{2} \llbracket \omega, \omega\rrbracket = \left(\sksym{3}\transp{\pr_\H}\right)(\d \omega) +\frac12 \,\widehat S \barwedge \curv_{\H}\,,
\end{equation}
and the theory of principal bundles, specifically within the context of \emph{fat bundles} \cite{wei1980advancesinmathematics}.

\begin{definition} Let $A \subset \mathfrak{g}^\ast$ be an $\mathrm{Ad}^\ast$-invariant subset. An \dfn{$A$-fat bundle} is a principal bundle $\pi \colon P \rightarrow M$ together with a principal connection $\alpha$ satisfying that $\mu \circ \omega_\alpha$ is non-degenerate on $\H = \Ker \alpha$, for every $\mu \in A$, where $\omega_\alpha = \left(\sksym{2}\transp{\pr_\H}\right)(\d \alpha)$ denotes the curvature of the connection. Equivalently, for every $\mu \in A$ and $p \in P$, the $2$-form $\mu \circ \omega_\alpha$ defines a linear symplectic form on the horizontal space $\H_p$.
\end{definition}

It should be clear that Equation~\eqref{eq:H_V_equation} simplifies greatly when $\omega = \mu \circ \omega_\alpha$ for some $A$-fat connection $\alpha$ and $\mu \in A$. Indeed, let $g$ be an (in principle, arbitrary) metric along the vertical distribution on $P$, namely $\V = \Ker \T\pi$. Then, for an $A$-fat connection, we have that 
\[
B_\mu \coloneqq \mu \circ \omega_\alpha + g
\]
is non-degenerate, for every $\mu \in A$. In fact, the decomposition it induces coincides with the decomposition defined by the principal connection:
\[
\T P = \Ker (\mu \circ \omega_\alpha) \oplus \Ker g = \V \oplus \Ker \alpha\,.
\]
Hence, we can ask whether $B_\mu$ induces a Jacobi (or Poisson) structure. Now, in light of the Bianchi identity, namely $\D_\alpha \D_\alpha  \alpha = 0$, where $\D_\alpha$ denotes the covariant derivative, we obtain that 
\[
\left(\sksym{3}\transp{\pr_\H}\right)(\d (\mu \circ \omega_\alpha)) = \D_\alpha (\mu \circ \omega_\alpha) = \mu \circ \D_\alpha \D_\alpha \alpha =  0
\]
and so Theorem \ref{th:big_theorem} implies that
$$
\frac{1}{2} \llbracket\mu \circ \omega_\alpha, \mu \circ \omega_\alpha \rrbracket = 
\frac12 \,\widehat g \barwedge \curv_{\H}\,.$$ Using the fact that $\omega_\alpha$ is, up to a sign, $\curv_\H$, we conclude the following:

\begin{proposition} Let $\pi \colon P \rightarrow M$ be an $A$-fat bundle, where $A \subset \mathfrak{g}^\ast$, together with a connection~$\alpha$. Then, for every $\mu \in A$, we have that $(\H, \mu \circ \omega_\alpha)$ is a symplectic vector bundle. Let $\Lambda_\mu$ denote the associated bivector field, extended to $\T P$  by employing the decomposition $\T P =\H \oplus \V$. We have the following:
\begin{enumerate}[\rm (i)]
    \item The bivector field $\Lambda_\mu$ never defines a Poisson structure on $P$.
    \item The bivector field $\Lambda_\mu$ can be completed to a Jacobi structure if, and only if, $\omega_\alpha$ (or $\curv_\H$) has a global image of dimension $1$. If this is the case, the Jacobi structure is independent of the choice of $\mu \in A$, up to a real scalar.
\end{enumerate}
\end{proposition}

\begin{proof} The first case is clear, as the curvature can never vanish (because the principal bundle is $A$-fat). For the second case, the necessity is immediate as, by Theorem \ref{thm:Jacobi_general}, we need the vector $\mathrm{curv}_{\mathcal{H}}(X_\mathcal{H}, Y_\mathcal{H})$ to be proportional to $R$, for every $X, Y \in \X(P)$.

Regarding the sufficiency, suppose that $\omega_\alpha$ has a $1$-dimensional image, say $\Ima \omega_\alpha = \left \langle\xi \right \rangle \subseteq \mathfrak{g}$, where $\xi \in \mathfrak{g}$. Then, necessarily $\mu (\xi) \neq 0$ (as $\alpha$ is $A$-fat and $\mu \in A$). In particular, 
\[
\curv_\H = (\mu \circ \omega_\alpha) \cdot \left(-\frac{\widehat{\xi}}{\mu(\xi)}\right)\,,
\]
where $\widehat{\xi}$ denotes the fundamental vector field induced on $P$. Then, by Theorem \ref{thm:Jacobi_general}, it is enough to show that, if $E \coloneqq - \widehat{\xi}/\mu(\xi)$, we have that $\left(\sksym{2}\transp{\pr_\H}\right)(\iota_{E} \d (\mu \circ \omega_\alpha)) = 0$ and that it leaves invariant the horizontal subspace. The latter condition follows from general theory of principal bundles, as $E$ is itself a fundamental vector field. For the former, let us first notice that we have
\[
\Lie_{\widehat{\xi}}\,\omega_\alpha = - \mathrm{ad}_\xi \circ \omega_\alpha\,,
\]
so that $\iota_E \d (\mu \circ \omega_\alpha) = \frac{1}{\mu(\xi)} \mu \circ \mathrm{ad}_\xi \circ \omega_\alpha$. However, since the image of $\omega_\alpha$ has rank one (proportional to $\xi$) we have $\mathrm{ad}_\xi \circ \omega_\alpha = 0$. To check uniqueness, it is enough to observe that for different $\mu$ and $\mu'$ we have the equality 
\[
\mu \circ \omega_\alpha = \frac{\mu(\xi)}{\mu'(\xi)} \mu' \circ \omega_\alpha\,,
\]
which finishes the proof.
\end{proof}

As a consequence, we have the following. 

\begin{corollary} Every fat $U(1)$-bundle has a canonical induced Jacobi structure.
\end{corollary}
In essence, every fat $U(1)$-bundle is a contact manifold, and the Jacobi structure above is the Jacobi structure induced by the contact form $\alpha$ (see \cite{Grabowska_2024}).

\section{Concluding remarks and outlook}
\label{section:conclusions}

In this paper, we compute the obstruction for a non-degenerate 2-covariant tensor field $B \in \Gamma(\T^\ast M \otimes \T^\ast M)$ to define a Poisson or Jacobi structure. This framework allows us to recover the classical brackets of symplectic, locally conformally symplectic, contact, cosymplectic, and cocontact geometry. 
As novel applications, we have characterised the conditions under which fat principal bundles and almost cosymplectic structures of order $p$ define Poisson or Jacobi brackets. 

This work opens several interesting lines of research:
\begin{enumerate}[\rm (i)]
    \item Building on the ideas presented here, a natural next step is to investigate whether the dynamical systems associated with the various geometries considered in this work can be unified using our framework.

    \item The case where the tensor field $B$ satisfies $\T M = \Ker \omega \oplus \Ker S$ falls under the general theory of $G$-structures defined by tensors. It would be interesting to relate the results in this text to the computation of the structure tensors of the associated $G$-structure. 
    
    \item Recently, there has been significant effort to study the brackets in the geometric formalisms of classical field theories. We believe that a similar computation would be of use to characterise under what conditions certain geometric structures induce a higher Poisson structure. 

    \item In several places in mechanics, namely when dealing with the Lagrangian point of view, rather than the Hamiltonian, it can happen that the natural tensor $B$ that one can build is \emph{degenerate}. We propose as further work to discuss Dirac (or Dirac--Jacobi) structures in this context.
\end{enumerate}

\subsection*{Acknowledgements}

The authors R. Izquierdo-López and  M. de León acknowledge financial support from the Spanish Ministry of Science, Innovation and Universities under grants PID2022-137909NB-C21 and the Severo Ochoa Program for Centers of Excellence in R\&D (CEX2023-001347-S). R. Izquierdo-López wishes to thank the Spanish Ministry of Science, Innovation and Universities for the contract FPU24/02636. Á. Martínez-Muñoz acknowledges financial support from a predoctoral contract funded by Universitat Rovira
i Virgili under grant 2025PMF-PIPF-14. 

We thank Jaime Bajo for carefully reading the manuscript and providing valuable comments and corrections.
\bigskip

\noindent {\bf Authors' contributions:} All authors contributed to the study conception and design. The manuscript was written and revised by all authors. All authors read and approved the final version.

\bigskip

\noindent {\bf Competing Interests:} The authors have no competing interests to declare. 

\bigskip

\noindent {\bf AI disclosure:} AI was used to assist with the review of the existing literature and to improve the language and clarity of the manuscript.

\bibliographystyle{abbrv}
{\small
\bibliography{references.bib}
}


\end{document}